\documentclass[final,hidelinks]{siamart251216}

\usepackage{amssymb}
\usepackage{mathrsfs}
\usepackage{mathtools}
\usepackage{graphicx}
\usepackage{booktabs}
\usepackage{algpseudocode}
\usepackage{float}
\usepackage{placeins}

\newsiamthm{assumption}{Assumption}
\crefname{assumption}{Assumption}{Assumptions}
\Crefname{assumption}{Assumption}{Assumptions}
\newsiamremark{remark}{Remark}
\newsiamremark{example}{Example}

\AddToHook{env/assumption/begin}{\crefalias{theorem}{assumption}}
\AddToHook{env/remark/begin}{\crefalias{theorem}{remark}}
\AddToHook{env/example/begin}{\crefalias{theorem}{example}}

\newcommand{\R}{\mathbb R}
\newcommand{\E}{\mathbb E}

\newcommand{\Res}{\mathcal R}
\newcommand{\Loss}{\mathcal L}

\DeclarePairedDelimiter{\norm}{\lVert}{\rVert}

\graphicspath{{figures/}}

\usepackage{microtype}
\AtBeginDocument{\microtypesetup{activate=false}}
\title{Deep Stochastic Residual Methods for Fully Nonlinear Parabolic Equations}

\author{Zhenhua Zhao\thanks{Peking University, Beijing, China
  (\email{zhenhuazhao@stu.pku.edu.cn}).}
  \and Jihao Long\thanks{Institute for Advanced Algorithms Research,
  Shanghai, China (\email{longjh1998@gmail.com}). Corresponding author.}}

\headers{Deep stochastic residual methods}{Z. Zhao and J. Long}

\hypersetup{
  pdftitle={\siampretitle Deep Stochastic Residual Methods for Fully Nonlinear Parabolic Equations},
  pdfauthor={Zhenhua Zhao; Jihao Long},
  pdfkeywords={stochastic residual methods, fully nonlinear parabolic equations,
    deep learning, backward regression, population error bounds,
    physics-informed neural networks}
}

\begin{document}
\maketitle

\begin{abstract}
We develop and analyze neural stochastic residual methods for fully
nonlinear parabolic terminal problems under exact diffusion sampling.
A quadratic stochastic residual supports global training
and backward regression with scalar networks and losses involving
only spatial derivatives. For global training, the correction removes
a Hessian penalty from the limiting first-order loss.
Under a uniform output regularity budget, we bound the
continuous residual norm by the square roots of the global loss or
normalized backward-regression excess loss, an $O(h^{1/2})$
discretization term for time step $h$, and terminal error. Residual stability then gives
value, gradient and Hessian error bounds. Under suitable solution and
terminal assumptions, we construct admissible networks with
$O(h^{1/2})$ bounds on these approximation errors and the square roots
of both loss criteria.
In ten-dimensional backward-regression tests using a shared schedule
tuned for the quadratic residual, mean relative Hessian errors are
$5$--$6\%$ with the quadratic residual and about $15\%$ with the
first-order residual. Tuned backward regression achieves
Hessian accuracy comparable to physics-informed neural networks across
the tested nonlinear coefficients. Translation tests show that the
ratio of Hessian error to residual can grow sharply under a fixed
sampling law.
\end{abstract}

\begin{keywords}
stochastic residual methods, fully nonlinear parabolic equations,
deep learning, backward regression, population error bounds,
physics-informed neural networks
\end{keywords}

\begin{MSCcodes}
65M12, 65M75, 35K55, 68T07
\end{MSCcodes}

\section{Introduction}
Parabolic partial differential equations (PDEs) with nonlinear
Hessian dependence arise in stochastic control when the diffusion
coefficient depends on the control \cite{fleming-soner-2006}.
Accurate Hessian approximation is therefore a central requirement
for their numerical solution. Diffusion sampling and neural
approximation offer an alternative to spatial grids in high dimensions.
For linear and semilinear equations, the underlying representations
come from Feynman--Kac formulas and backward stochastic differential
equations (BSDEs)
\cite{le-gall-2016,pardoux-peng-1992}.

For semilinear equations, the Deep BSDE method learns gradients
through forward simulation and terminal matching
\cite{e-han-jentzen-2017,han-jentzen-e-2018}.
A posteriori estimates control its solution error through the
terminal loss and time discretization error \cite{han-long-2020}.
Physics-informed neural networks (PINNs)
\cite{raissi-perdikaris-karniadakis-2019} and the Deep Galerkin Method
(DGM) \cite{sirignano-spiliopoulos-2018} instead minimize differential
residuals and boundary or terminal errors. Stability estimates
link these residual losses to solution errors
\cite{mishra-molinaro-2023}.

For fully nonlinear equations, Hessian approximation and error control
pose an additional challenge. Second-order BSDE representations
\cite{cheridito-soner-touzi-victoir-2007} and probabilistic numerical schemes
\cite{fahim-touzi-warin-2011} provide foundations for this task.
Deep 2BSDE couples value and gradient dynamics and learns
second-order processes by terminal matching~\cite{beck-e-jentzen-2019}.
Deep backward dynamic programming (DBDP) learns values and gradients
by local regressions \cite{hure-pham-warin-2020}; its fully nonlinear
extension also approximates the Hessian \cite{pham-warin-germain-2021}.
Other approaches use global
stochastic regression through Deep Picard iteration
\cite{han-hu-long-zhao-2026} or derivative labels from perturbed paths
\cite{jia-ouyang-pham-zhou-2026}.

We relate stochastic training losses to the continuous PDE residual
loss and apply stability from the companion analysis \cite{long-zhao-2026}
to control value, gradient and Hessian errors. We assume exact diffusion
sampling, with losses and errors measured under the same law.
After the setup in \cref{sec:problem}, our contributions are as follows.

\emph{(i) Algorithmic formulation.}
\Cref{sec:residual-methods} constructs a quadratic stochastic residual
that approximates the time-integrated PDE residual by subtracting
spatial Taylor increments. Global training (G) uses one space--time
network; backward regression (BR) fits spatial networks against fixed
next layers. Both use only spatial derivatives.
The correction removes a Hessian penalty from the limiting first-order
global loss and couples the current Hessian to quadratic target
fluctuations in BR.

\emph{(ii) Error control and neural approximation.}
Under a uniform output regularity budget,
\cref{sec:residual-theory} bounds the continuous residual norm by
the square roots of the G loss or normalized BR excess loss,
an $O(h^{1/2})$ discretization term for time step $h$, and terminal error.
Residual stability then controls value, gradient and Hessian errors.
For solutions with finite energy and suitable terminal data,
we construct two-hidden-layer tanh networks satisfying the budget
with $O(h^{1/2})$ bounds on the graph-norm error and square roots of
both loss criteria, up to terminal approximation error.
We also give PDE conditions for finite solution energy.
These are population estimates; network complexity and training
guarantees remain open. \Cref{app:remaining-proofs} collects the
remaining proofs and auxiliary lemmas.

\emph{(iii) Numerical comparisons.}
\Cref{sec:experiments} compares Hessian accuracy and training time
on ten-dimensional problems. Under a shared BR-tuned schedule,
the quadratic residual improves Hessian accuracy over the first-order
residual. Tuned BR has accuracy comparable to PINNs and better than G,
with lower mean training times across the tested coefficients.
Time-grid refinement tests the predicted $O(h^{1/2})$ Hessian-error rate.
Translation tests show that residual-to-Hessian control can
deteriorate under a fixed sampling law.
\Cref{app:experimental-settings} provides implementation details and
additional results.

\section{Problem setting and PINNs}\label{sec:problem}
Let $T>0$ and let $\mathbb S^d$ denote the symmetric $d\times d$ matrices.
For a scalar function $v$, write $J^2v=(v,Dv,D^2v)$, where $D=D_x$.
We consider the terminal value problem
\begin{equation}\label{eq:terminal-problem}
 \Res(u^\star):= \partial_tu^\star+\mathscr H(t,x,J^2u^\star)=0
 \quad\text{on }(0,T)\times\R^d,\qquad u^\star(T,\cdot)=g.
\end{equation}
Here $\mathscr H:[0,T]\times\R^d\times\R\times\R^d\times\mathbb S^d\to\R$;
for $j=(z,p,\Gamma)$, set $|j|=|z|+|p|+|\Gamma|$.

Fix an initial probability law $\mu_0$. Let $W$ be a $d$-dimensional
Brownian motion and take $X_0\sim\mu_0$ independent of $W$.
Let $(\mathcal F_t)$ be the usual augmented filtration generated
by $X_0$ and $W$.
The sampling diffusion is
\begin{equation}\label{eq:sampling-sde}
 dX_t=b(t,X_t)\,dt+\sigma(t,X_t)\,dW_t,\qquad \sigma=a^{1/2}.
\end{equation}
All matrix norms are Frobenius norms, and $A_1:A_2=\operatorname{tr}(A_1^\top A_2)$.
Set $\mu_t=\operatorname{Law}(X_t)$ and $\nu(dt,dx)=dt\,\mu_t(dx)$.
For a finite measure $\mu$ with a positive density, set
\begin{equation*}
 H^k(\mu)=\{\phi\in H^k_{\mathrm{loc}}(\R^d):\norm{\phi}_{H^k(\mu)}<\infty\}, \quad
 \norm{\phi}_{H^k(\mu)}^2=\sum_{j=0}^k\int_{\R^d}|D^j\phi|^2\,d\mu,\,\, k \in \mathbb N.
\end{equation*}
Here $D^0\phi=\phi$. We impose the following sampling and residual conditions.
\begin{assumption}[Sampling and residual regularity]\label{ass:standard}
The matrix $a$ is symmetric, and there exist $\lambda,\Lambda,B,L>0$ such that
$\lambda I\preceq a(t,x)\preceq\Lambda I,\, |b(t,x)|\le B$ and
\begin{equation*}
 |a(t,x)-a(s,y)|+|b(t,x)-b(s,y)|
 \le L\bigl(|x-y|+|t-s|^{1/2}\bigr).
\end{equation*}
There is $L_{\mathscr H}>0$ such that, for all $t,s\in[0,T]$, $x,y\in\R^d$
and $j,\widetilde j\in\R\times\R^d\times\mathbb S^d$,
\begin{equation*}
 |\mathscr H(t,x,j)-\mathscr H(s,y,\widetilde j)|
\le L_{\mathscr H}\bigl[|j-\widetilde j|
       +(1+|j|+|\widetilde j|)(|x-y|+|t-s|^{1/2})\bigr].
\end{equation*}
The data satisfy $\mathscr H(\cdot,\cdot,0)\in L^2(\nu)$ and $g\in H^1(\mu_T)$.
\end{assumption}
Additional PDE assumptions for stability and finite solution energy
are stated in \cref{sec:pinn-theory,thm:pde-regularity}.
Under \cref{ass:standard}, $\sigma$ inherits the spatial and temporal
regularity of $a$, and \cref{eq:sampling-sde} has a unique strong
solution for every initial law \cite[Theorems~8.3 and~8.5(iii)]{le-gall-2016}.
We assume access to exact marginal and transition samples of $X$;
the estimates below do not include SDE simulation error.
We use the occupation norm and the parabolic graph norm
\begin{equation*}
 \norm{f}_\nu^2=\int_0^T\E|f(t,X_t)|^2\,dt,\qquad
 \norm{v}_{\mathbb X_\nu}^2
 =\norm{v}_\nu^2+\norm{v_t}_\nu^2+\norm{Dv}_\nu^2+\norm{D^2v}_\nu^2.
\end{equation*}
The space $\mathbb X_\nu$ consists of functions with finite graph norm
and weak derivatives. Under \cref{ass:standard}, the terminal trace
$v\mapsto v(T)$ is bounded from $\mathbb X_\nu$ to $H^1(\mu_T)$
\cite[Theorem~3.6]{long-zhao-2026}.

\subsection{PINN formulation}\label{sec:pinn-formulation}
The continuous residual loss serves both as the PINN objective and
as the reference for the stochastic error bounds.
Fix a function class $\mathcal V\subset\mathbb X_\nu$ containing
the solution $u^\star$ and the outputs under consideration.
Define the interior loss, terminal penalty and PINN loss with a terminal weight $\lambda_T>0$ by
\begin{equation*}
 \Loss_{\mathrm{int}}(v)=T^{-1}\norm{\Res(v)}_\nu^2,
 \quad
 \mathcal T(v)=\lambda_T\norm{v(T)-g}_{H^1(\mu_T)}^2,
 \quad
 \Loss(v)=\Loss_{\mathrm{int}}(v)+\mathcal T(v).
\end{equation*}
The $H^1$ penalty matches the terminal norm in the stability estimates.

PINN outputs $v_\theta\in\mathcal V$ use a scalar network
$\mathcal N_\theta$ with parameters $\theta$. Networks are $C^3$
in their inputs; space--time networks also belong to $\mathbb X_\nu$.
With a soft terminal, $v_\theta=\mathcal N_\theta$ and $\mathcal T$ penalizes
the terminal mismatch. To enforce a prescribed terminal datum $\widehat g$, fixed during training, use the hard representation
\begin{equation}\label{eq:hard-terminal}
 v_\theta(t,x)=\alpha(t)\widehat g(x)
                 +\beta(t)\mathcal N_\theta(t,x), \qquad\alpha(t) = t/T, \quad\beta(t)=1-t/T.
\end{equation}
For this hard representation, $\widehat g\in C^2(\R^d)$ and
$\alpha\widehat g\in\mathbb X_\nu$ suffice for pointwise residual
evaluation and a finite PINN loss under \cref{ass:standard}.
One may take $\widehat g=g$ when these conditions hold for the
original datum, or use a smooth approximation $\widehat g_h$
and retain its terminal error in the analysis.

With either representation, PINN training minimizes a sample
approximation of $\Loss(v_\theta)$, using automatic differentiation
to evaluate $\Res(v_\theta)$. For a hard terminal,
$\mathcal T(v_\theta)$ is independent of $\theta$, so the optimized
objective reduces to $\Loss_{\mathrm{int}}(v_\theta)$.
Interior samples use $\tau$ uniform on $(0,T)$ and $X_\tau\sim\mu_\tau$;
the soft terminal penalty uses samples from $\mu_T$ and evaluations
of $g$ and its weak gradient.

\subsection{Continuous stability}\label{sec:pinn-theory}
We control the graph-norm error through a residual stability estimate
uniform over $v\in\mathcal V$:
\begin{equation}\label{eq:pinn-stability}
 \norm{v-u^\star}_{\mathbb X_\nu}^2
 \le C_{\mathrm{stab}}\Loss(v)^\vartheta,\qquad 0<\vartheta\le1.
\end{equation}
Here $C_{\mathrm{stab}}$ may depend on $\mathcal V$, the PDE,
the sampling law, $\lambda_T$ and $\vartheta$, but is independent of
$v\in\mathcal V$.
We summarize three sufficient criteria from \cite{long-zhao-2026},
all under \cref{ass:standard}.

\emph{(i) Small Hessian perturbations.}
Choose a reference principal coefficient $A$ and a Hessian weight
$q$, both bounded, symmetric, uniformly
elliptic and uniformly continuous in $(t,x)$.
Write $P_A=-\partial_t-A:D^2$ and
$\mathcal H_qw=q^{1/2}D^2w\,q^{1/2}$; $A$ may differ from $a/2$.
For $\norm f_{\kappa,\nu}:=\norm{e^{-\kappa(T-t)}f}_\nu$,
$\kappa\ge0$, define the zero-terminal Hessian gain by
\begin{equation}\label{eq:weighted-gains}
 C_2(\kappa)
 =\sup_{\substack{0\ne w\in\mathbb X_\nu\\w(T)=0}}
       \frac{\norm{\mathcal H_qw}_{\kappa,\nu}}
            {\norm{P_Aw}_{\kappa,\nu}},\qquad
 K_A=\lim_{\kappa\to\infty}C_2(\kappa).
\end{equation}
For finite gains, $K_A=G_0((0,T))$ in the notation of
\cite[Proposition~3.7]{long-zhao-2026}, with no additional time weight.
Write $N_A(t,x,j)=\mathscr H(t,x,j)-A(t,x):\Gamma$ and assume
\begin{equation}\label{eq:hessian-perturbation}
 |N_A(t,x,z,p,\Gamma)-N_A(t,x,\widetilde z,\widetilde p,\widetilde\Gamma)|
\le L_0|z-\widetilde z|+L_1|p-\widetilde p|
       +L_2|q^{1/2}(\Gamma-\widetilde\Gamma)q^{1/2}|,
\end{equation}
with $L_0,L_1,L_2\ge0$. If $K_A<\infty$ and $L_2K_A<1$,
\cite[Proposition~5.1]{long-zhao-2026} gives a unique solution in
$\mathbb X_\nu$ and \cref{eq:pinn-stability} with
$\mathcal V=\mathbb X_\nu$ and $\vartheta=1$.

In the matching case $A=a/2$, normalize the Hessian by $q=a$.
Every initial law gives
$2\le K_{a/2}\le2\sqrt2$ \cite[Theorem~4.9]{long-zhao-2026}, with
\begin{equation*}
 K_{a/2}=\begin{cases}
 2\sqrt2,&\mu_0\text{ is a point mass},\\
 2,&\mu_0\text{ is a finite mixture of nondegenerate Gaussian laws}.
 \end{cases}
\end{equation*}
For general $A$, a sufficient condition for $K_A<\infty$ is
\begin{equation}\label{eq:reference-gap}
 A(t,x)\preceq\chi a(t,x),\qquad 0<\chi<1,
 \quad (t,x)\notin\mathcal K,
\end{equation}
where $\mathcal K$ is empty or $[\tau,T]\times\overline B_R$, with
$0<\tau<T$ and $0<R<\infty$
\cite[Corollary~4.8]{long-zhao-2026}.

\emph{(ii) Strict upper gaps on regular classes.}
For this criterion and (iii), assume $\mathcal V$ is a bounded convex subset of
$C_b^{1+\gamma/2,2+\gamma}$, $0<\gamma<1$.
The exponents specify time and space regularity, respectively;
$b$ indicates boundedness of the full H\"older norm.
We use $C_b^s$ for the spatial version, with $s=1$ interpreted
as Lipschitz continuity.
Sufficient conditions for this regularity of $u^\star$ are given in
\cite[Proposition~5.8]{long-zhao-2026}; a common bound on the
outputs is imposed separately.
Fix a ball containing all values $J^2v(t,x)$ with $v\in\mathcal V$.
On a neighborhood of this ball, assume $\mathscr H$ is
$C^1$ in $j$ with bounded first derivatives, and
$\mathscr H_\Gamma$ is uniformly elliptic.
We further assume that $\mathscr H_\Gamma$ is jointly uniformly
continuous in $(t,x,j)$, with $(t,x)\in[0,T]\times\R^d$ and $j$ in
this ball.

Suppose every $v\in\mathcal V$ satisfies
\begin{equation}\label{eq:nonlinear-spectral-gap}
 \mathscr H_\Gamma(t,x,J^2v(t,x))\preceq\chi a(t,x),
 \qquad 0<\chi<1,\quad (t,x)\notin\mathcal K,
\end{equation}
where $\chi,\mathcal K$ are fixed over $\mathcal V$ and $\mathcal K$ is
as in \cref{eq:reference-gap}. Then \cite[Theorem~5.4]{long-zhao-2026}
gives \cref{eq:pinn-stability} with $\vartheta=1$.
The regularity and ellipticity conditions on $\mathscr H$ also
apply inside $\mathcal K$.
For constant $a$, $b=0$, $\mu_0=\delta_0$ and $g=0$, the linear example
$\mathscr H(t,x,j)=\chi a:\Gamma$ shows that $\vartheta=1$ fails when
$\chi\ge1$, even on a bounded H\"older ball
\cite[Example~5.7]{long-zhao-2026}.

\emph{(iii) Interpolation beyond the strict upper gap.}
Under the regularity assumptions in (ii), allow any $\chi>0$ in
\cref{eq:nonlinear-spectral-gap}. Then
\cite[Theorem~5.6]{long-zhao-2026} gives \cref{eq:pinn-stability}
for every $0<\vartheta<\min\{1,\chi^{-1}\}$.
Both the error and residual remain measured under the original law $\nu$.

\section{Stochastic residual methods}\label{sec:residual-methods}
\subsection{Stochastic residuals and losses}\label{sec:residual-formulation}
Use a uniform grid $t_i=ih$ with $Nh=T$, $N\ge2$ and $h\le1$.
Set $X_i=X_{t_i}$ and $\Delta X_i=X_{i+1}-X_i$.
Fix an interpolation weight $\varpi\in[0,1]$;
we sometimes omit its superscript in proofs.
For a jet function $F(t,x,z,p,\Gamma)$ and current and next spatial
functions $\phi,\psi$, define
\begin{equation*}
 [F]_i^\varpi(\phi;\psi)
 =(1-\varpi)F(t_i,X_i,J^2\phi(X_i))
   +\varpi F(t_{i+1},X_{i+1},J^2\psi(X_{i+1})).
\end{equation*}
All quantities are evaluated at each endpoint before averaging.
Write $\mathscr H_i^\varpi(\phi;\psi)=[\mathscr H]_i^\varpi(\phi;\psi)$.
The weights $\varpi=0$ and $1$ select the current (C) and next (N)
endpoints, respectively.

The first-order residual subtracts the linear spatial increment
and compensates for diffusion:
\begin{equation}\label{eq:first-order-residual}
 \widehat Z_i^\varpi(\phi;\psi)
 =\psi(X_{i+1})-\phi(X_i)-D\phi(X_i)\cdot\Delta X_i
  +h[\mathscr H-\tfrac12a:\Gamma]_i^\varpi(\phi;\psi).
\end{equation}
The quadratic residual $Z_i^\varpi(\phi;\psi)$ subtracts both
spatial Taylor increments:
\begin{equation}\label{eq:quadratic-residual}
 \psi(X_{i+1})-\phi(X_i)-D\phi(X_i)\cdot\Delta X_i
 -\tfrac12D^2\phi(X_i):(\Delta X_i\Delta X_i^\top)
       +h\mathscr H_i^\varpi(\phi;\psi).
\end{equation}
Taylor derivatives are always evaluated at the current point.
First-order and quadratic refer to the spatial increments, not to
orders of convergence in time. The driver and diffusion compensation
in the first-order residual can still involve Hessians.

For a space--time function $v$, write
$Z_i^\varpi(v)=Z_i^\varpi(v(t_i,\cdot);v(t_{i+1},\cdot))$,
and use the same convention for $\widehat Z_i^\varpi$,
$\mathscr H_i^\varpi$ and the local losses below. We also write
$f_i=f(t_i,X_i)$ for a field $f$, $D^kv_i=(D^kv)(t_i,X_i)$, and
$\mathscr H_v(t,x)=\mathscr H(t,x,J^2v(t,x))$.
The local and global population losses are
\begin{gather*}
 \ell_i^\varpi(\phi;\psi)=\E|Z_i^\varpi(\phi;\psi)|^2,\quad
 \Loss_{h,\mathrm{int}}^\varpi(v)=\frac1{Th}\sum_{i<N}\ell_i^\varpi(v), \quad
 \Loss_h^\varpi(v)=\Loss_{h,\mathrm{int}}^\varpi(v)+\mathcal T(v).
\end{gather*}
Hats on $\ell$ and $\Loss$ denote the corresponding losses with
$\widehat Z$ in place of $Z$.
Thus the interior loss averages $\E|Z_i^\varpi(v)/h|^2$ over the
$N$ time cells. Under \cref{ass:standard}, these losses are finite
whenever $\mathcal T(v)<\infty$ and $J^2v(t_i,X_i)$,
$0\le i\le N$, are square integrable. We impose this node
integrability condition on the free networks. For the hard representation
\cref{eq:hard-terminal}, a sufficient additional condition is
$\E|J^2(\alpha\widehat g)(t_i,X_i)|^2<\infty$ for $0\le i\le N$.

For smooth $v$ with bounded derivatives, the Taylor subtraction gives
\begin{equation*}
 Z_i^\varpi(v)
 =\int_{t_i}^{t_{i+1}}\Res(v)(s,X_s)\,ds+O_{L^2}(h^{3/2}),
\end{equation*}
where $L^2$ is with respect to the sampling probability.
Thus $\Loss_{h,\mathrm{int}}^\varpi$ approximates the continuous interior loss;
\cref{sec:residual-theory} gives uniform estimates.
For $Z_i^\varpi$, the sampling coefficients enter only through
the diffusion samples.

\subsection{Global training and backward regression}\label{sec:training}
Each residual supports either training procedure. We use G and BR
for the quadratic residual $Z$, and F1-G and F1-BR for the same
procedures with $\widehat Z$. Suffixes C and N indicate $\varpi=0$
and $\varpi=1$, respectively. Sample averages replace expectations
during training.

\emph{Global training (G).}
G uses one space--time output $v_\theta$, shared across all time cells,
with either terminal representation in \cref{sec:pinn-formulation}.
Training minimizes $\Loss_h^\varpi(v_\theta)$ by differentiating
with respect to the shared network parameters through both endpoint
evaluations in every cell. For a hard terminal, the constant penalty
$\mathcal T(v_\theta)$ drops from the objective.

\emph{Backward regression (BR).}
For a hard terminal, set $V_N=\widehat g$ and use
\begin{equation}\label{eq:br-hard-terminal}
 V_i(x)=\alpha(t_i)\widehat g(x)
          +\beta(t_i)\mathcal N_{\theta_i}(x),\qquad i<N.
\end{equation}
BR fits layers in the order $i=N-1,\ldots,0$. At step $i$, it
minimizes $\ell_i^\varpi(V_i;V_{i+1})$ over $\theta_i$ while holding
$V_{i+1}$ fixed. The selected $V_i$ is then frozen as the next target.
Networks share an architecture; each fit after the first starts from
the selected next-layer parameters. \Cref{app:br-settings} describes
the initialization.

A soft-terminal variant first fits $V_N=\mathcal N_{\theta_N}$
to $g$ with the $H^1(\mu_T)$ penalty and freezes it. The same
backward regressions then use $V_i=\mathcal N_{\theta_i}$ for $i<N$.

\subsection{Role of the quadratic correction}\label{sec:quadratic-correction}
The diffusion compensation in $\widehat Z_i^\varpi$ follows the
chosen driver endpoint, whereas the quadratic term in $Z_i^\varpi$
uses the current Hessian. Their difference is
\begin{equation}\label{eq:first-order-difference}
 \widehat Z_i^\varpi(v)-Z_i^\varpi(v)
 =\tfrac12D^2v_i:(\Delta X_i\Delta X_i^\top-ha_i)
 -\tfrac{h\varpi}{2}
       \bigl[(a:D^2v)_{i+1}-(a:D^2v)_i\bigr].
\end{equation}
Under Brownian sampling, $\Delta X_i=\Delta W_i:=W_{t_{i+1}}-W_{t_i}$,
so the first term is a conditionally centered quadratic fluctuation.
It is the entire difference for C; for N, there is also an increment
of the diffusion compensation.

For global training, the next result identifies the limiting
first-order loss and shows that its extra Hessian term can displace
the PDE solution, even for linear equations.

\begin{proposition}
\label{prop:linear-correction-limit}
Under \cref{ass:standard}, let
$v\in C_b^{1+\gamma/2,2+\gamma}$ for some $0<\gamma<1$.
Then
\begin{equation}\label{eq:linear-correction-limit}
 \widehat{\Loss}_h^\varpi(v)\longrightarrow
 \widehat{\Loss}(v):=\Loss(v)
       +\frac1{2T}\norm{\sigma^\top D^2v\,\sigma}_\nu^2
 \qquad\text{as }h\to0.
\end{equation}
If $u^\star$ belongs to this space and
$\norm{\sigma^\top D^2u^\star\sigma}_\nu>0$, it is not a local
minimizer of $\widehat{\Loss}$, even under the constraint $v(T)=g$.
\end{proposition}
\begin{proof}
Set $r=\Res(v)$ and $\Delta W_i=W_{t_{i+1}}-W_{t_i}$.
By \cref{ass:standard},
$\Delta X_i=\sigma_i\Delta W_i+O_{L^4}(h)$.
Parabolic Taylor expansion then gives, uniformly over $i$ for this fixed $v$,
\begin{equation*}
 \widehat Z_i(v)=hr_i+Q_i+O_{L^2}(h^{1+\gamma/2}),\quad
 Q_i=\tfrac12(\sigma_i^\top D^2v_i\sigma_i):
                    (\Delta W_i\Delta W_i^\top-hI).
\end{equation*}
Conditional Gaussian moments imply
\begin{equation*}
 \E[Q_i\mid\mathcal F_{t_i}]=0,\qquad
 \E[Q_i^2\mid\mathcal F_{t_i}]
       =\tfrac12h^2|\sigma_i^\top D^2v_i\sigma_i|^2.
\end{equation*}
Orthogonality and Cauchy--Schwarz give
\begin{equation*}
 \widehat{\Loss}_h(v)-\mathcal T(v)
 =\frac hT\sum_{i<N}\E\bigl[|r_i|^2
          +\tfrac12|\sigma_i^\top D^2v_i\sigma_i|^2\bigr]
       +O(h^{\gamma/2}).
\end{equation*}
Since $r$ and $\sigma^\top D^2v\sigma$ are continuous in $L^2$
along $X$, Riemann-sum convergence proves \cref{eq:linear-correction-limit}.
The same Taylor expansion gives
$Z_i^\varpi(v)=hr_i+O_{L^2}(h^{1+\gamma/2})$, and hence
$\Loss_h^\varpi(v)\to\Loss(v)$.

For the bias, set
$v_\delta(t,x)=[1-\delta(T-t)]u^\star(t,x)$.
Then $v_\delta(T)=g$ and $v_\delta\to u^\star$ in the stated
H\"older space. Since $\Res(u^\star)=0$ and
$\mathcal T(v_\delta)=0$, \cref{ass:standard} gives
$\Loss(v_\delta)=O(\delta^2)$. Hence
\begin{equation*}
 \widehat{\Loss}(v_\delta)-\widehat{\Loss}(u^\star)
 =-\frac{\delta}{T}
       \norm{\sqrt{T-t}\,\sigma^\top D^2u^\star\sigma}_\nu^2
       +O(\delta^2).
\end{equation*}
The weighted norm is positive, so the difference is negative for
all sufficiently small $\delta>0$.
For any such fixed $\delta$, \cref{eq:linear-correction-limit} also
gives $\widehat{\Loss}_h(v_\delta)<\widehat{\Loss}_h(u^\star)$
for all sufficiently small $h$.
\end{proof}

Backward regression freezes the next layer $\psi$ at each fit.
Its terms in \cref{eq:first-order-residual} therefore form a fixed
target. Conditionally on $X_i$, the current value and compensated
driver terms are constant, and the gradient term is linear in $\Delta X_i$.
Projecting the target onto this conditional affine space splits
the local loss into a squared fitting error and the squared norm
of an orthogonal remainder. The remainder depends only on $\psi$,
so it cannot affect the current fit. In global training, the next
layer also varies with the parameters, so this remainder is no
longer constant. Thus \cref{prop:linear-correction-limit} does not
establish the same bias for F1-BR.

The semilinear DBDP theory illustrates this distinction.
For Brownian sampling and
$\mathscr H(t,x,z,p,\Gamma)=\tfrac12\operatorname{tr}\Gamma+f(t,x,z,p)$,
F1-BR-C has the DBDP2 form, with gradients obtained by differentiating
the value network. Under the assumptions of
\cite[Section~3.2 and Theorem~4.2]{hure-pham-warin-2020}, its convergence
estimates control the value and diffusion-weighted gradient, but
not the Hessian. The fully nonlinear extension introduces a Hessian
approximation but does not establish convergence
\cite[Section~2]{pham-warin-germain-2021}.
For general diffusions, comparison with DBDP also requires matching
the sampling and drift discretizations.

With the quadratic residual, the current Hessian also multiplies a
quadratic increment, so BR can fit the target's quadratic fluctuations
through this coefficient. \Cref{sec:first-order-experiment} compares
the resulting Hessian accuracy.

\section{Population error analysis}\label{sec:residual-theory}
We now pass from the fixed-function limits in
\cref{sec:quadratic-correction} to bounds uniform over G and BR
outputs depending on $h$, then establish neural attainability.
Combining these bounds with \cref{eq:pinn-stability} controls the
solution error.
Throughout, assume \cref{ass:standard} and use the grid
of \cref{sec:residual-formulation}. Constants here and in
\cref{app:energy-estimates} may vary; their dependence on
$d,T,\lambda_T,\lambda,\Lambda,\allowbreak B,L,L_{\mathscr H}$ is implicit.
Other dependence is stated.

\subsection{Main results}\label{sec:main-results}
The loss estimates require one additional spatial derivative and
control of variation between grid nodes. Define
\(
 E(v)=\norm v_{\mathbb X_\nu}+\norm{Dv}_{\mathbb X_\nu}
\)
with the vector graph norm defined componentwise.
For $I_i=(t_i,t_{i+1})$ and a function $f(t,x)$, define
\begin{equation*}
 \omega_h^X(f)=\sum_{i<N}\int_{I_i}\E\bigl[
 |f(t,X_t)-f(t_i,X_i)|^2+|f(t,X_t)-f(t_{i+1},X_{i+1})|^2\bigr]dt.
\end{equation*}
For $r_t^v=\Res(v)(t,X_t)$, set
\begin{gather}
 c_i(v)=\E\left[\frac1h\int_{I_i}r_t^vdt
                           \,\middle|\,\mathcal F_{t_i}\right],
 \qquad A_h(v)=\sum_{i<N}\int_{I_i}\E|r_t^v-c_i(v)|^2dt,
 \label{eq:weak-grid-budget}\\
 \mathfrak E_h(v)^2=E(v)^2+h^{-1}\bigl[\omega_h^X(D^2v)+A_h(v)\bigr].\notag
\end{gather}
Here $\omega_h^X(D^2v)$ measures Hessian variation relative to both
cell endpoints, and $A_h$ measures residual variation about the
conditional cell average $c_i(v)$. The budget applies to the
complete output, including its terminal branch.
On cells $I_i$ with $i\ge1$, \cref{eq:positive-grid-variation}
already bounds the Hessian variation by $ChE(v)^2$;
the initial-cell contribution requires separate control.

\begin{assumption}[Uniform output regularity budget]\label{ass:energy-budget}
For each $h$, $\mathcal M_h$ is a class of
$v\in C^{0,2}([0,T]\times\R^d)$.
Here $C^{0,2}$ means that $v,Dv,D^2v$ are continuous in $(t,x)$.
There is $M<\infty$, independent of $h$, such that
\(
 \mathfrak E_h(v)\le M
\)
for any $v\in\mathcal M_h$.
\end{assumption}

\emph{Population loss bounds.}
For G, the discrete objective directly approximates the continuous loss.
Here $C_M$ may depend on $M$ and
$R_{\mathscr H}:=\norm{\mathscr H(\cdot,\cdot,0)}_\nu$.

\begin{theorem}[G output bound]\label{thm:network-errors}
Under \cref{ass:energy-budget}, for every $v\in\mathcal M_h$,
\begin{equation}\label{eq:full-loss-comparison}
 \left|\sqrt{\Loss_h^\varpi(v)}-\sqrt{\Loss(v)}\right|
 \le C_M\sqrt h.
\end{equation}
\end{theorem}

For BR,
define the excess loss against each fixed next layer by
\begin{equation}
 \eta_i^\varpi(v)
 =\ell_i^\varpi(v)
       -\inf_{\phi\in\mathcal S_i}\ell_i^\varpi(\phi;v(t_{i+1},\cdot)),
\end{equation}
where $\mathcal S_i=\{\phi\in C^2(\R^d):J^2\phi\in L^2(\mu_{t_i})\}$.
This nonnegative population excess loss includes approximation and
optimization errors relative to $\mathcal S_i$.

\begin{theorem}[BR output bound]\label{thm:br-error}
Under \cref{ass:energy-budget}, suppose the BR layers have a common
realization $V_i=v(t_i,\cdot)$ with $v\in\mathcal M_h$. Then
\begin{equation}\label{eq:br-loss-bound}
 \Loss(v)\le C_M\left[h+\frac1{Th}\sum_{i<N}\eta_i^\varpi(v)\right]
                     +\mathcal T(v).
\end{equation}
\end{theorem}

The two bounds give $\sqrt{\Loss(v)}=O(h^{1/2})$ when
$\sqrt{\Loss_h^\varpi(v)}=O(h^{1/2})$, or when the square roots of
the normalized BR excess and terminal penalty are both $O(h^{1/2})$.
If the output classes $\mathcal M_h$ lie in a common class
$\mathcal V$ satisfying \cref{eq:pinn-stability}, either bound gives
$\norm{v-u^\star}_{\mathbb X_\nu}=O(h^{\vartheta/2})$,
hence $O(h^{1/2})$ for $\vartheta=1$.

\emph{Neural attainability.}
Finite solution energy, verified under PDE conditions in
\cref{sec:finite-energy}, yields networks with controlled budgets and losses.

\begin{theorem}[Quantitative neural attainability]\label{thm:attainability}
Suppose $E(u^\star)<\infty$. Consider either output representation:
\begin{enumerate}
\item Soft terminal: $v_h=\mathcal N_h$, with $\delta_h=0$.
\item Hard terminal: $v_h=\alpha\widehat g_h+\beta\mathcal N_h$,
with $\delta_h=\norm{\widehat g_h-g}_{H^2(\mu_T)}$, where
$\widehat g_h\in C^3(\R^d)\cap H^3(\bar\mu)$ is prescribed and
$\bar\mu(dx)=\int_0^T\mu_t(dx)\,dt$.
\end{enumerate}
Set $V_i=v_h(t_i,\cdot)$ for BR.
For each sufficiently small $h=T/N$, there are finite-width tanh
networks $\mathcal N_h$ with two hidden layers such that $\mathfrak E_h(v_h)^2\le C_*(1+\delta_h^2)$ and
\begin{equation}\label{eq:joint-loss-attainability}
 \norm{v_h-u^\star}_{\mathbb X_\nu}^2+\Loss(v_h)
 +\Loss_h^\varpi(v_h)+\frac1{Th}\sum_{i<N}\eta_i^\varpi(v_h)
 \le C_*(h+\delta_h^2).
\end{equation}
Moreover, $E(v_h-u^\star)\to0$ if $\delta_h\to0$.
Here $C_*$ depends on $E(u^\star)$.
\end{theorem}

Here $g\in H^2(\mu_T)$ by \cref{eq:energy-trace}.
For hard terminals, the approximation error and square roots of the
losses are $O(h^{1/2})$ when $\delta_h=O(h^{1/2})$;
taking $\widehat g_h=g$ gives $\delta_h=0$ if $g$ satisfies the
stated smoothness and integrability assumptions.

The next corollary gives the common H\"older bound for
\cref{sec:pinn-theory}\textup{(ii)--(iii)}.

\begin{corollary}\label{cor:holder-attainability}
Under the hypotheses of \cref{thm:attainability}, suppose
$\norm{u^\star}_\gamma<\infty$, where
$\norm f_\gamma=\norm f_{C_b^{1+\gamma/2,2+\gamma}([0,T]\times\R^d)}$
and $0<\gamma<1$.
For hard terminals, assume additionally that
$\sup_h\norm{\widehat g_h}_{C_b^{2+\gamma}}
\le C(d,\gamma)\norm g_{C_b^{2+\gamma}}$.
Then, for some $C_\gamma$ depending on $\gamma$, the outputs in
\cref{thm:attainability} can also satisfy
\begin{equation}\label{eq:holder-attainability}
 \sup_h\norm{v_h}_\gamma\le C_\gamma(1+\norm{u^\star}_\gamma).
\end{equation}

\end{corollary}

\subsection{Residual consistency and the G bound}
Write $\mathcal G=\partial_t+b\cdot D+\tfrac12a:D^2$ for the
sampling generator, acting componentwise on vector fields.
Along the sampling paths, write $b_t=b(t,X_t)$,
$a_t=a(t,X_t)$ and $\sigma_t=\sigma(t,X_t)$.
\Cref{ass:standard} gives the residual and increment bounds:
\begin{gather}
 \norm{\Res(v)}_\nu
 \le C\norm v_{\mathbb X_\nu}+R_{\mathscr H}
 \le C(E(v)+R_{\mathscr H}),\label{eq:residual-growth}\\
 \E[|X_t-X_s|^p\mid\mathcal F_s]\le C_p(t-s)^{p/2},
 \qquad 0 \le s \le t \le T, \quad p=2,4.\label{eq:sampling-moments}
\end{gather}
\eqref{eq:sampling-moments} uses the conditional Burkholder--Davis--Gundy
 inequality.

\begin{lemma}\label{lem:grid-regularity}
Let $v\in C^{0,2}([0,T]\times\R^d)$ satisfy
$E(v)+\omega_h^X(D^2v)<\infty$. Then
\begin{equation}\label{eq:energy-grid-bound}
\begin{gathered}
 \omega_h^X(v)+\omega_h^X(Dv)\le ChE(v)^2,\quad
 h\sum_{i=0}^N\E|J^2v(t_i,X_i)|^2
      \le C[E(v)^2+\omega_h^X(D^2v)],\\
 \omega_h^X(\mathscr H_v)
      \le C[h(1+E(v)^2)+\omega_h^X(D^2v)].
\end{gathered}
\end{equation}
\end{lemma}
\begin{proof}
For $f=v$ and each component of $Dv$, finite energy gives
$f,f_t,Df,D^2f\in L^2(\nu)$ and hence local square integrability
with respect to $dt\,dx$ on $(0,T)\times\R^d$, since $d\mu_t/dx$ is
bounded away from zero on compact sets \cite[Theorems~7 and~10]{aronson-1968}.
In divergence form, $\mathcal G=\partial_t+\tfrac12\operatorname{div}(aD)
+(b-\tfrac12\operatorname{div}a)\cdot D$ has bounded drift.
After localization, $\mathcal Gf\in L^2$, so
\cite[Theorem~4.3(i)]{klimsiak-2013} applies with zero divergence
source and gives, for $0<s\le t\le T$,
\begin{equation*}
 f(t,X_t)-f(s,X_s)=\int_s^t\mathcal Gf(r,X_r)dr
                          +\int_s^t Df(r,X_r)\sigma_r dW_r.
\end{equation*}
Since $\norm{\mathcal Gf}_\nu+\norm{Df\,\sigma}_\nu\le CE(v)$, continuity of
$f(t,X_t)$ extends the identity to $s=0$. Consequently,
\begin{equation}\label{eq:ito-variation-bound}
 \omega_h^X(f)\le 2h\norm{Df\sigma}_\nu^2
                         +2h^2\norm{\mathcal Gf}_\nu^2.
\end{equation}
This proves $\omega_h^X(v)+\omega_h^X(Dv)\le ChE(v)^2$,
including the initial cell.
Comparing $f(t_i,X_i)$ with $f(t,X_t)$ in the adjacent cells gives
\begin{equation*}
 h\sum_{i=0}^N\E|f(t_i,X_i)|^2
       \le4\norm f_\nu^2+2\omega_h^X(f).
\end{equation*}
Applying this inequality to $v,Dv,D^2v$ bounds
$h\sum_{i=0}^N\E|J^2v(t_i,X_i)|^2$ as in \cref{eq:energy-grid-bound}.

To bound $\omega_h^X(\mathscr H_v)$, set $j_t=J^2v(t,X_t)$ and
$Y_t=\mathscr H(t,X_t,j_t)$.
For $0\le s\le t\le T$, insert $\mathscr H(t,X_t,j_s)$
between $Y_t$ and $Y_s$. Since $j_s$ is $\mathcal F_s$-measurable,
\cref{ass:standard,eq:sampling-moments} give
\begin{equation}\label{eq:nonlinear-increment-bound}
 \E|Y_t-Y_s|^2\le C\E|j_t-j_s|^2
                         +C(t-s)\E(1+|j_s|^2).
\end{equation}
For the grid estimates, put $j_i=j_{t_i}$ and $Y_i=Y_{t_i}$.
For the right endpoint write $Y_t-Y_{i+1}=(Y_t-Y_i)+(Y_i-Y_{i+1})$ and use
$|j_{i+1}-j_i|^2\le2|j_t-j_i|^2+2|j_t-j_{i+1}|^2$.
Integration yields
\begin{equation*}
 \omega_h^X(\mathscr H_v)
 \le C\sum_{f\in\{v,Dv,D^2v\}}\omega_h^X(f)
                  +Ch^2\sum_{i<N}\E(1+|j_i|^2).
\end{equation*}
The preceding bounds and $h\le1$ give the last inequality in
\cref{eq:energy-grid-bound}.
\end{proof}

The conditional average $c_i(v)$ minimizes
$\int_{I_i}\E|r_t^v-c|^2dt$ over square-integrable
$\mathcal F_{t_i}$-measurable random variables $c$.
Orthogonal projection gives
\begin{equation}\label{eq:projection-contraction}
 A_h(v)+h\sum_{i<N}\E|c_i(v)|^2=\norm{\Res(v)}_\nu^2.
\end{equation}

The next lemma gives the residual comparisons for G and BR.

\begin{lemma}
Let $v$ be as in \cref{lem:grid-regularity}. Set
$r_t=\Res(v)(t,X_t)$ and
$r_{h,v}^\varpi(t)=Z_i^\varpi(v)/h$ on $I_i$. For a stochastic process $Y$, write
$\norm{Y}_2^2:=\int_0^T\E|Y_t|^2\,dt$. Then
\begin{equation}\label{eq:common-residual-bound}
 \frac1h\sum_{i<N}\E|Z_i^\varpi(v)-hc_i(v)|^2
       +\norm{r_{h,v}^\varpi-r}_2^2
\le C\bigl[h(1+E(v)^2)+\omega_h^X(D^2v)+A_h(v)\bigr].
\end{equation}
\end{lemma}
\begin{proof}
Set $H_t=D^2v(t,X_t)$ and
\begin{equation*}
 R_s^i=Dv(s,X_s)-Dv(t_i,X_i)-H_{t_i}(X_s-X_i),\quad
 Q_i=h\mathscr H_i(v)-\int_{I_i}\mathscr H_v(s,X_s)ds.
\end{equation*}
It\^o's formula for $v$ and the subtracted quadratic increment gives
\begin{equation}\label{eq:cell-residual-identity}
 Z_i(v)-\int_{I_i}r_sds
 =Q_i+\int_{I_i}[b_s\cdot R_s^i+\tfrac12a_s:(H_s-H_{t_i})]ds
                   +\int_{I_i}(R_s^i)^\top\sigma_s dW_s.
\end{equation}
The process $R_s^i$ satisfies
\begin{equation*}
 dR_s^i=(\mathcal GDv(s,X_s)-H_{t_i}b_s)ds
                            +(H_s-H_{t_i})\sigma_s dW_s.
\end{equation*}
Cauchy--Schwarz, It\^o isometry and the bound for
$h\sum_i\E|H_{t_i}|^2$ in \cref{eq:energy-grid-bound} yield
\begin{equation*}
 \sum_{i<N}\int_{I_i}\E|R_s^i|^2ds
 \le C[h^2E(v)^2+h\omega_h^X(D^2v)].
\end{equation*}
Since $Q_i$ is a convex combination of the two endpoint
quadrature errors,
\begin{equation*}
 h^{-1}\sum_{i<N}\E|Q_i|^2
                  \le\omega_h^X(\mathscr H_v).
\end{equation*}
In \cref{eq:cell-residual-identity}, use these bounds for
$R_s^i,Q_i$ and \cref{eq:energy-grid-bound} for
$\omega_h^X(\mathscr H_v)$. This gives
\begin{equation*}
 \frac1h\sum_{i<N}\E\left|Z_i(v)-\int_{I_i}r_sds\right|^2
 \le C[h(1+E(v)^2)+\omega_h^X(D^2v)].
\end{equation*}
Comparing through $c_i=c_i(v)$ and applying Cauchy--Schwarz gives
\begin{equation*}
 \frac1h\sum_{i<N}\E|Z_i(v)-hc_i|^2
       +\norm{r_{h,v}-r}_2^2
 \le\frac Ch\sum_{i<N}\E\left|Z_i(v)-\int_{I_i}r_sds\right|^2
       +CA_h(v),
\end{equation*}
which proves \cref{eq:common-residual-bound}.
\end{proof}

\begin{proof}[Proof of \cref{thm:network-errors}]
Under \cref{ass:energy-budget}, the right side of
\cref{eq:common-residual-bound} is at most $C_Mh$.
Since $T^{-1}\norm{r_{h,v}^\varpi}_2^2
=\Loss_{h,\mathrm{int}}^\varpi(v)$, the reverse triangle inequality gives
\begin{equation*}
 \left|\sqrt{\Loss_{h,\mathrm{int}}^\varpi(v)}-\sqrt{\Loss_{\mathrm{int}}(v)}\right|
 \le C_M\sqrt h.
\end{equation*}
Adding the same nonnegative terminal penalty cannot increase this
root difference, proving \cref{eq:full-loss-comparison}.
Squaring
$\sqrt{\Loss_{\mathrm{int}}(v)}\le\sqrt{\Loss_{h,\mathrm{int}}^\varpi(v)}+C_M\sqrt h$
and adding $\mathcal T(v)$ give
\begin{equation}\label{eq:global-loss-bound}
 \Loss(v)\le 2\Loss_{h,\mathrm{int}}^\varpi(v)+C_Mh+\mathcal T(v).
\end{equation}
\end{proof}

\subsection{The BR population bound}
With the next layer fixed, we seek a smooth correction of the
current layer close to $hc_i(v)$. The Markov property makes
$c_i(v)$ a function of $X_i$, but it need not be $C^2$ as required
by $\mathcal S_i$. We therefore smooth the residual forward in time
using fixed nonnegative probability kernels
$\rho\in C_c^\infty((1/2,1))$ and
$\psi\in C_c^\infty(B(0,1))$:
\begin{equation}\label{eq:forward-smoothing}
 S_\delta F(t,x)=\int_{1/2}^1\int_{\R^d}
 \rho(s)\psi(y)F(t+\delta s,x+\sqrt\delta\,y)\,dy\,ds,
 \quad 0\le t\le T-\delta.
\end{equation}

\begin{proof}[Proof of \cref{thm:br-error}]
For a fixed output, the definition of the excess losses gives
\begin{equation*}
 \Loss_{h,\mathrm{int}}(v)=\frac1{Th}\sum_{i<N}\eta_i(v)
       +\frac1{Th}\sum_{i<N}\inf_{\phi\in\mathcal S_i}
                     \ell_i(\phi;v(t_{i+1},\cdot)).
\end{equation*}
To bound the second sum, set
$r_t=\Res(v)(t,X_t)$, $c_i=c_i(v)$ and
$A_i=\int_{I_i}\E|r_t-c_i|^2dt$.

Under \cref{ass:standard}, the transition density $p(s,x;t,y)$ of
$X$ satisfies Aronson's lower bound
\cite[Theorems~7 and~10]{aronson-1968}. In particular,
\begin{equation}\label{eq:local-kernel-lower-bound}
 p(t,x;t+r,x+y)\ge c\delta^{-d/2},
 \quad \delta/2\le r\le\delta,\quad |y|\le\sqrt\delta.
\end{equation}
Set $\zeta_i=hS_h\Res(v)(t_i,\cdot)$ and
$\phi_i=v(t_i,\cdot)+\zeta_i$.
Below, values of $\zeta_i$ and its spatial derivatives are taken at $X_i$.
Conditioning on $X_i$ and changing variables in
\cref{eq:local-kernel-lower-bound} give
\begin{equation*}
 \E\int_{1/2}^1\int_{B(0,1)}
 |\Res(v)(t_i+hs,X_i+\sqrt h\,y)-c_i|^2dy\,ds
 \le CA_i/h.
\end{equation*}
Compute $D\zeta_i(X_i)$ and $D^2\zeta_i(X_i)$ by differentiating the
spatial kernel $\psi$. Since $\int D\psi=\int D^2\psi=0$,
subtracting $c_i$ inside these integrals does not change their values.
Cauchy--Schwarz and the preceding bound yield
\begin{equation}\label{eq:local-correction-bounds}
 \E[|\zeta_i-hc_i|^2+h|D\zeta_i|^2+h^2|D^2\zeta_i|^2]\le ChA_i,
 \qquad \E|\zeta_i|^2\le2h^2\E|c_i|^2+ChA_i.
\end{equation}
Since $\phi_i\in C^2$, the bounds
\cref{eq:local-correction-bounds,eq:energy-grid-bound} give
$\phi_i\in\mathcal S_i$.

Substitute $\phi_i=v(t_i,\cdot)+\zeta_i$ into $Z_i$ with the next
layer fixed.
Using \cref{eq:sampling-moments}, the Lipschitz bound in
\cref{ass:standard} and then \cref{eq:local-correction-bounds}, we obtain
\begin{equation*}
\begin{aligned}
 \ell_i(\phi_i;v(t_{i+1},\cdot))
 &\le C\E\bigl[|Z_i(v)-hc_i|^2+|\zeta_i-hc_i|^2+h|D\zeta_i|^2
                    +h^2|D^2\zeta_i|^2+h^2|\zeta_i|^2\bigr]\\
 &\le C\E|Z_i(v)-hc_i|^2+ChA_i+Ch^4\E|c_i|^2.
\end{aligned}
\end{equation*}
Summing, dividing by $Th$, and using
\cref{eq:residual-growth,eq:common-residual-bound,eq:projection-contraction}
give
\begin{equation}
 \frac1{Th}\sum_{i<N}\inf_{\phi\in\mathcal S_i}
          \ell_i(\phi;v(t_{i+1},\cdot))
 \le C\bigl[h(1+E(v)^2+R_{\mathscr H}^2)
                          +\omega_h^X(D^2v)+A_h(v)\bigr].
\end{equation}
Under \cref{ass:energy-budget}, this is at most $C_Mh$, so
$\Loss_{h,\mathrm{int}}(v)\le(Th)^{-1}\sum_{i<N}\eta_i(v)+C_Mh$.
Substituting into \cref{eq:global-loss-bound} proves
\cref{eq:br-loss-bound}.
\end{proof}

\begin{remark}
For $T=1$, $X_t=W_t$, $u_t+\tfrac12\Delta u=0$ with $g=0$ and $h = 1/N$, take
\(
 w_h(t,x)=h^{3/2}\sin(2\pi t/h)\sin(x_1/\sqrt h).
\)
Then $\mathfrak E_h(w_h)\le C$,
$\Loss_h^\varpi(w_h)=\eta_i^\varpi(w_h)=0$, and
$\Loss(w_h)\asymp\norm{w_h}_{\mathbb X_\nu}^2\asymp h$.
Thus the mesh terms in \cref{eq:full-loss-comparison,eq:br-loss-bound}
and the squared graph-norm rate at $\vartheta=1$ are sharp
for continuous realizations.
The projection term $A_h$ is also necessary:
$v_h(t,x)=h\sin(2\pi t/h)$ has bounded energy, zero Hessian
variation and zero discrete loss, but $\Loss(v_h)=2\pi^2$.
\end{remark}

\subsection{Quantitative neural attainability}
We first regularize the solution, then approximate the resulting
target by networks in energy and at grid nodes. The construction
uses the following estimates for linear solutions, traces and jet increments.
Write $\mathscr L_{a/2}(f,\tau)$ for the matching solution of
$-\partial_tz-\tfrac12a:D^2z=f$, $z(T)=\tau$.
Under \cref{ass:standard},
\cite[Theorem~3.6 and Corollary~4.8]{long-zhao-2026} give
\begin{equation}\label{eq:matching-linear-bound}
 \norm{\mathscr L_{a/2}(f,\tau)}_{\mathbb X_\nu}
 \le C(\norm f_\nu+\norm\tau_{H^1(\mu_T)}).
\end{equation}
A subscript $I\subset(0,T)$ restricts the time integrals in
$\norm\cdot_\nu$, $\norm\cdot_{\mathbb X_\nu}$ and $E$ to $I$.

Applying \cite[Theorem~3.6(1)]{long-zhao-2026} to $v,Dv$ gives
\begin{equation}\label{eq:energy-trace}
 \norm{v(T)}_{H^2(\mu_T)}\le CE(v).
\end{equation}
For each $s\in(0,T]$, applying the same trace theorem on $[0,s]$ gives
$v(s)\in H^2(\mu_s)$ with norm at most $C_sE(v)$.
We use this trace to define $J^2v(s,X_s)\in L^2$; convergence in
energy implies convergence of these random variables for each fixed $s>0$.

\begin{lemma}
Let $E(v)<\infty$ and $j_t=J^2v(t,X_t)$ for $t>0$.
For $0<\tau\le T/4$,
\begin{align}
 \sum_{i=1}^{N-1}\int_{I_i}\E[|j_t-j_{t_i}|^2+|j_t-j_{t_{i+1}}|^2]dt
 &\le ChE(v)^2,\label{eq:positive-grid-variation}\\
 \int_0^{T-\tau}\E|j_{t+\tau}-j_t|^2dt
 &\le C\tau E(v)^2.\label{eq:jet-time-translation}
\end{align}
\end{lemma}

\begin{proof}
By \cref{lem:energy-density}, it suffices to take $v$ smooth and
compactly supported: energy convergence controls the integral terms,
and \cref{eq:energy-trace} on $[0,r]$ gives convergence at each $r>0$.
Both assertions follow from the local estimate
\begin{equation}\label{eq:right-endpoint-increment}
 \int_s^r\E|j_t-j_r|^2dt\le C(r-s)E_{(s,r)}(v)^2,
 \qquad 0\le s<r\le T.
\end{equation}
For the components $v,Dv$, this follows from the It\^o estimate
used in \cref{eq:ito-variation-bound}, restricted to $(s,r)$.
To estimate the Hessian components, fix $I=(s,r)$, put $\delta=r-s$,
and set $w=Dv$ and $Q=(\partial_t+\tfrac12a:D^2)w$.
By \cref{ass:standard},
$\norm{D^2w}_{\nu,I}^2+\norm Q_{\nu,I}^2\le CE_I(v)^2$.
Define $e=\mathscr L_{a/2}(-\mathbf1_IQ,0)$ componentwise and set $z=w-e$.
Then $e=0$ on $[r,T]$, and $z$ satisfies
$(\partial_t+\tfrac12a:D^2)z=0$ on $I$, with $z(r)=w(r)$.
We bound $\norm{De}_{\nu,I}$ and apply It\^o's formula to $Dz$
to estimate the increments of $D^2v=Dz+De$.
\Cref{eq:matching-linear-bound} and the
short-interval energy estimate \cite[Lemma~A.1(1)]{long-zhao-2026}
give
\begin{equation*}
 \norm{D^2e}_{\nu,I}^2+\delta^{-1}\norm{De}_{\nu,I}^2
 \le C\norm Q_{\nu,I}^2.
\end{equation*}
For $z$, the differentiated It\^o formula
\cite[proof of Lemma~A.2]{long-zhao-2026} has drift
\begin{equation*}
 \mathcal G(D_jz)=b\cdot D(D_jz)-\tfrac12(D_ja):D^2z.
\end{equation*}
By \cref{ass:standard}, the drift above and the diffusion
$D^2z\,\sigma$ are bounded by $C|D^2z|$ along $X$.
Cauchy--Schwarz and It\^o
isometry therefore give
\begin{equation*}
 \int_s^r\E|Dz(t,X_t)-Dz(r,X_r)|^2dt
 \le C\delta\norm{D^2z}_{\nu,I}^2.
\end{equation*}
Since $De(r)=0$ and $\norm{D^2z}_{\nu,I}^2\le CE_I(v)^2$,
these estimates prove \cref{eq:right-endpoint-increment}.

For $i\ge1$, averaging over the preceding cell and applying
\cref{eq:right-endpoint-increment} give
\begin{equation*}
 h\E|j_{t_{i+1}}-j_{t_i}|^2
 \le ChE_{(t_{i-1},t_{i+1})}(v)^2.
\end{equation*}
Using \cref{eq:right-endpoint-increment} on $I_i$ and comparing
$j_t$ with $j_{t_i}$ through $j_{t_{i+1}}$ gives
\cref{eq:positive-grid-variation} after summation; the intervals
$(t_{i-1},t_{i+1})$ overlap at most twice.

Partition $(0,T-\tau)$ into the nonempty intervals
$J_k=(k\tau,\min\{(k+1)\tau,T-\tau\})$.
For $t\in J_k$, compare $j_t$ and $j_{t+\tau}$ with $j_{r_k}$,
where $r_k=\min\{(k+2)\tau,T\}$.
Apply \cref{eq:right-endpoint-increment} on $(k\tau,r_k)$ and sum over
$k$; these intervals overlap at most twice, giving
\cref{eq:jet-time-translation}.
\end{proof}

\begin{proof}[Proof of \cref{thm:attainability}]
\emph{Regularization with terminal value $g$.}
Write $u=u^\star$ and smooth near $t=0$ to obtain an initial
Hessian in $L^2(\mu_0)$. Fix a smooth cutoff $\chi=1$ on $[0,T/4]$
and $\chi=0$ on $[T/2,T]$. For $h<T/4$, set
\begin{equation}\label{eq:quantitative-target}
 U_h=\chi S_hu+(1-\chi)u,\qquad 0\le t\le T/2.
\end{equation}
For $t\ge T/2$, set $U_h=u$, so $U_h(T)=g$. We will prove
\begin{equation}\label{eq:target-estimates}
\begin{gathered}
 E(U_h)\le CE(u),\qquad E(U_h-u)\to0,\\
 \norm{U_h-u}_{\mathbb X_\nu}^2+\norm{\Res(U_h)}_\nu^2
 +\omega_h^X(D^2U_h)+A_h(U_h)\le C_*h.
\end{gathered}
\end{equation}
The energy assertions follow from the $L^2$ boundedness and strong
convergence of $S_h$ in \cref{app:energy-estimates}: $S_h$ commutes
with the derivatives in $E$, and multiplication by $\chi$ is bounded
in energy.

To bound $\norm{U_h-u}_{\mathbb X_\nu}^2$, we need squared $L^2$ smoothing errors
of order $h$ for $u,Du,D^2u$ and $u_t$.
For $F\in L^2(\nu)$ and a.e.\ $t\in(0,T-h)$,
Jensen's inequality and \cref{eq:local-kernel-lower-bound} give
\begin{equation}\label{eq:smoothing-increment-bound}
 \E|S_hF(t,X_t)-F(t,X_t)|^2
 \le\frac Ch\int_{h/2}^h
       \E|F(t+r,X_{t+r})-F(t,X_t)|^2\,dr.
\end{equation}
Apply this with $F=J^2u$ and $F=\mathscr H_u$ and integrate over
$t\in(0,T/2)$. By \cref{eq:jet-time-translation,eq:nonlinear-increment-bound},
\begin{equation*}
 \norm{S_hJ^2u-J^2u}_{\nu,(0,T/2)}^2
 +\norm{S_h\mathscr H_u-\mathscr H_u}_{\nu,(0,T/2)}^2\le C_*h.
\end{equation*}
The equation $u_t=-\mathscr H_u$ gives
\(
 \partial_t(S_hu-u)=-(S_h\mathscr H_u-\mathscr H_u).
\)
The cutoff product rule now yields
$\norm{U_h-u}_{\mathbb X_\nu}^2\le C_*h$.
Since $\Res(u)=0$, the Lipschitz bound in \cref{ass:standard} and
\cref{eq:projection-contraction} imply
\begin{equation*}
 A_h(U_h)\le\norm{\Res(U_h)}_\nu^2
 \le C\norm{U_h-u}_{\mathbb X_\nu}^2\le C_*h.
\end{equation*}

For $\omega_h^X(D^2U_h)$, \cref{eq:positive-grid-variation}
already bounds the contribution of cells $I_i$ with $i\ge1$ by
$ChE(U_h)^2$.
On the first cell, $U_h=S_hu$ is smooth, and
\cref{eq:forward-smoothing-marginal} gives
$J^2U_h(0)\in L^2(\mu_0)$.
To apply \cref{eq:ito-variation-bound} to its Hessian, we estimate
$D^3S_hu$ and $\mathcal GD^2S_hu$.
Set $H=D^2u$. Integration by parts in the time-shift variable,
with $\int\rho'=0$, gives
\begin{equation*}
 \partial_tS_hH(t,x)
 =-h^{-1}\int\rho'(s)\psi(y)
 [H(t+hs,x+\sqrt h\,y)-H(t,x)]\,dy\,ds.
\end{equation*}
Repeating the argument for \cref{eq:smoothing-increment-bound}
with $|\rho'|$ in place of $\rho$, then using
\cref{eq:jet-time-translation}, gives
$\norm{\partial_tS_hH}_{\nu,(0,T/2)}^2\le Ch^{-1}E(u)^2$.
For spatial derivatives, $DS_hH=S_hDH$. Moving one further
derivative onto $\psi$ and using \cref{eq:forward-smoothing-marginal} give
\begin{equation*}
 \norm{DS_hH}_{\nu,(0,T/2)}^2
 +h\norm{D^2S_hH}_{\nu,(0,T/2)}^2\le C\norm{DH}_\nu^2.
\end{equation*}
Since $b,a$ are bounded, these estimates yield
\begin{equation}\label{eq:smoothed-generator}
 \norm{D^3S_hu}_{\nu,(0,T/2)}^2
 +h\norm{\mathcal GD^2S_hu}_{\nu,(0,T/2)}^2\le CE(u)^2.
\end{equation}
By \cref{eq:ito-variation-bound,eq:smoothed-generator}, the $i=0$
contribution to $\omega_h^X(D^2U_h)$ is also at most $C_*h$,
completing \cref{eq:target-estimates}.

\emph{Prescribed terminal values.}
For soft terminals set $\xi_h=0$. For hard terminals, put
$z_h=\mathscr L_{a/2}(0,\widehat g_h-g)$, $\xi_h=(1-\chi)z_h$.
The target $\widetilde U_h=U_h+\xi_h$ then has terminal value $g$
in the soft case and $\widehat g_h$ in the hard case.
We will prove
\begin{equation}\label{eq:terminal-target-estimates}
 \mathfrak E_h(\widetilde U_h)^2\le C_*(1+\delta_h^2), \quad
 \norm{\widetilde U_h-u}_{\mathbb X_\nu}^2
       +\norm{\Res(\widetilde U_h)}_\nu^2\le C_*(h+\delta_h^2).
\end{equation}
To bound $E(z_h)$, the difference-quotient argument of
\cite[proof of Lemma~A.2]{long-zhao-2026} gives, for smooth
compactly supported terminal data, $z_h,Dz_h\in\mathbb X_\nu$ and
\begin{equation}\label{eq:matching-gradient-equation}
 P_{a/2}D_jz_h=\tfrac12(D_ja):D^2z_h,\qquad
 D_jz_h(T)=D_j(\widehat g_h-g),\quad 1\le j\le d.
\end{equation}
Applying \cref{eq:matching-linear-bound} to $z_h$ and its derivatives,
and using the bound on $Da$ in \cref{ass:standard}, yields
\begin{equation}\label{eq:matching-energy}
 E(z_h)\le C\norm{\widehat g_h-g}_{H^2(\mu_T)}=C\delta_h.
\end{equation}
Approximation of the terminal data in $H^2(\mu_T)$, with this bound
applied to solution differences, extends
\cref{eq:matching-gradient-equation,eq:matching-energy} to the stated data.

The cutoff product rule gives $E(\xi_h)\le C\delta_h$ and hence
$E(\widetilde U_h-u)\le E(U_h-u)+C\delta_h$.
Since $\Res(u)=0$, \cref{ass:standard,eq:target-estimates} give
\begin{equation*}
 \norm{\widetilde U_h-u}_{\mathbb X_\nu}
 +\norm{\Res(\widetilde U_h)}_\nu
 \le C\norm{U_h-u}_{\mathbb X_\nu}+C\delta_h
 \le C_*(\sqrt h+\delta_h).
\end{equation*}

Since $\xi_h=0$ on $[0,T/4]$, \cref{eq:positive-grid-variation}
also gives $\omega_h^X(D^2\xi_h)\le Ch\delta_h^2$.
In view of \cref{eq:target-estimates}, the bound on
$\mathfrak E_h(\widetilde U_h)$ therefore reduces to
$A_h(\widetilde U_h)\le C_*h(1+\delta_h^2)$.
For each of $U_h,\xi_h,\widetilde U_h$, \cref{lem:energy-density}
gives smooth approximations converging in $E$, with their values,
gradients and Hessians converging in $L^2(\mu_{t_i})$ at every
grid node. Passing to the limit therefore extends
\cref{eq:energy-grid-bound} to these functions.
Write
\begin{equation*}
 \Res(\widetilde U_h)=\Res(U_h)+\partial_t\xi_h
                  +\mathscr H_{\widetilde U_h}-\mathscr H_{U_h}.
\end{equation*}
Approximate $\Res(U_h)$ by zero and
$\partial_t\xi_h,\mathscr H_{\widetilde U_h},\mathscr H_{U_h}$
by their left-endpoint values. The minimizing property of $c_i$ gives
\begin{equation*}
 A_h(\widetilde U_h)
 \le C\bigl[\norm{\Res(U_h)}_\nu^2
       +\omega_h^X(\partial_t\xi_h)
       +\omega_h^X(\mathscr H_{\widetilde U_h})
       +\omega_h^X(\mathscr H_{U_h})\bigr].
\end{equation*}
\Cref{eq:target-estimates} gives $\norm{\Res(U_h)}_\nu^2\le C_*h$,
while \cref{eq:energy-grid-bound} bounds
$\omega_h^X(\mathscr H_{\widetilde U_h})+\omega_h^X(\mathscr H_{U_h})$
by $C_*h(1+\delta_h^2)$.
To control $\omega_h^X(\partial_t\xi_h)$, use the equation for $z_h$:
\begin{equation*}
 \partial_t\xi_h=-\chi'z_h-\tfrac12a:D^2\xi_h.
\end{equation*}
This identity also specifies $\partial_t\xi_h$ at the grid nodes.
Apply \cref{eq:ito-variation-bound} to $\chi'z_h$.
For $a:D^2\xi_h$, \cref{ass:standard,eq:sampling-moments}
and the node bound in \cref{eq:energy-grid-bound} give
\begin{equation*}
 \omega_h^X(a:D^2\xi_h)
 \le C\omega_h^X(D^2\xi_h)
       +Ch^2\sum_{i<N}\E|D^2\xi_h(t_i,X_i)|^2
 \le Ch\delta_h^2.
\end{equation*}
Thus $\omega_h^X(\partial_t\xi_h)\le Ch\delta_h^2$ and hence
$A_h(\widetilde U_h)\le C_*h(1+\delta_h^2)$.
This completes \cref{eq:terminal-target-estimates}.

\emph{Network realization.}
For fixed $h$, measure energy and node errors by
\begin{equation}\label{eq:grid-density-topology}
 d_h(f,F)^2=E(f-F)^2+
 h\sum_{i=0}^N\norm{J^2(f-F)(t_i)}_{L^2(\mu_{t_i})}^2.
\end{equation}
We seek a network output $v_h$ satisfying
\begin{equation}\label{eq:network-grid-accuracy}
 d_h(v_h,\widetilde U_h)^2\le h^2.
\end{equation}
For soft terminals, use \cref{lem:energy-density} to choose a smooth
compactly supported $F_h$ approximating $\widetilde U_h$ in $d_h$,
and set $Q_h=F_h$.
For hard terminals, apply \cref{lem:energy-density} to
$\widetilde U_h-\alpha\widehat g_h$, choosing a smooth approximation
$F_h$ in $d_h$ with $F_h(T)=0$.
The condition $\widehat g_h\in H^3(\bar\mu)$ gives
$E(\widetilde U_h-\alpha\widehat g_h)<\infty$.
Since $U_h$ is smooth near $t=0$, $\xi_h=0$ there and
$\widehat g_h\in C^3$, the function
$J^2(\widetilde U_h-\alpha\widehat g_h)$ is locally continuous there.
At $t=0$, it equals $J^2U_h(0)\in L^2(\mu_0)$ because $\alpha(0)=0$.
Since $F_h(T)=0$ and $\beta(t)=(T-t)/T$, the quotient $Q_h=F_h/\beta$
extends smoothly across $t=T$.
In either case $Q_h$ has a compactly supported smooth extension
to $\R^{d+1}$. To approximate it globally, let $\Psi$ apply
$\tanh$ to every coordinate.
The function $Q_h\circ\Psi^{-1}$ has a smooth zero extension to
the closed cube $[-1,1]^{d+1}$.
By \cite[Theorem~3]{hornik-1991}, a one-hidden-layer tanh network
$P_h$ approximates it in $C^3$ to any prescribed accuracy.
The chain rule and bounded derivatives of $\Psi$ give
\begin{equation*}
 \mathcal N_h=P_h\circ\Psi,\qquad
 \norm{\mathcal N_h-Q_h}_{C_b^3(\R^{d+1})}\le\varepsilon_h.
\end{equation*}
Since
$d_h(\mathcal N_h,Q_h)+d_h(\beta\mathcal N_h,\beta Q_h)\le C\varepsilon_h$,
choosing the smooth approximation error and $\varepsilon_h$
sufficiently small gives \cref{eq:network-grid-accuracy} for
$v_h=\mathcal N_h$ or $v_h=\alpha\widehat g_h+\beta\mathcal N_h$,
respectively.

\emph{Transfer to the population losses.}
For $e_h=v_h-\widetilde U_h$, the definitions of $\omega_h^X$ and $d_h$,
\cref{eq:network-grid-accuracy,ass:standard} and
\cref{eq:projection-contraction} give
\begin{equation*}
 \omega_h^X(D^2e_h)\le Ch^2, \quad
 \norm{\Res(v_h)-\Res(\widetilde U_h)}_\nu^2\le Ch^2,\quad
 A_h(v_h)\le2A_h(\widetilde U_h)
          +Ch^2.
\end{equation*}
Together with $E(e_h)\le h$ and \cref{eq:terminal-target-estimates},
these imply $\mathfrak E_h(v_h)^2\le C_*(1+\delta_h^2)$ and
$\norm{v_h-u}_{\mathbb X_\nu}^2\le C_*(h+\delta_h^2)$.
Moreover, $E(v_h-u)\le h+E(U_h-u)+C\delta_h\to0$ if $\delta_h\to0$.
By \cref{eq:energy-trace,eq:network-grid-accuracy},
$\mathcal T(v_h)\le Ch^2$ for soft terminals;
$\mathcal T(v_h)\le\lambda_T\delta_h^2$ for hard terminals.
Thus $\Loss(v_h)\le C_*(h+\delta_h^2)$.
\Cref{eq:common-residual-bound} gives
\begin{equation*}
 \Loss_h^\varpi(v_h)
 \le2\Loss(v_h)+Ch\bigl[1+\mathfrak E_h(v_h)^2\bigr]
 \le C_*(h+\delta_h^2).
\end{equation*}
Finally, $0\le(Th)^{-1}\sum_i\eta_i^\varpi(v_h)
\le \Loss_{h,\mathrm{int}}^\varpi(v_h)\le\Loss_h^\varpi(v_h)$ proves
\cref{eq:joint-loss-attainability}.
\end{proof}

\begin{proof}[Proof of \cref{cor:holder-attainability}]
The positive kernels in $S_h$ and the fixed cutoff in
\cref{eq:quantitative-target} give
$\sup_h\norm{U_h}_\gamma\le C_\gamma\norm{u^\star}_\gamma$.
For hard terminals, \cref{ass:standard} controls the
$C_b^{\gamma/2,\gamma}$ norm of $a$. The global parabolic Schauder
estimate \cite[Theorem~5.1.10]{lunardi-1995}, applied after time reversal,
then gives
\begin{equation*}
 \norm{z_h}_\gamma
 \le C_\gamma\norm{\widehat g_h-g}_{C_b^{2+\gamma}}
 \le C_\gamma\norm{u^\star}_\gamma,
\end{equation*}
where the last inequality uses the assumed bound on
$\widehat g_h$ and $g=u^\star(T)$.
Thus $\sup_h\norm{\widetilde U_h}_\gamma
\le C_\gamma\norm{u^\star}_\gamma$ in both cases.
When choosing $F_h$ above, use \cref{eq:holder-smooth-density}
to ensure that the smooth output $F_h$ or $\alpha\widehat g_h+F_h$
has H\"older norm at most $C_\gamma\norm{u^\star}_\gamma$.
The difference between $v_h$ and this smooth output is
$\mathcal N_h-Q_h$ or $\beta(\mathcal N_h-Q_h)$, respectively.
Its H\"older norm is at most $C_\gamma\varepsilon_h$ by the
$C_b^3$ bound. The triangle inequality and $\varepsilon_h\le1$ give
\cref{eq:holder-attainability} while preserving \cref{eq:network-grid-accuracy}.
\end{proof}

\subsection{Sufficient conditions for finite solution energy}
\label{sec:finite-energy}
We now verify the finite solution energy required by
\cref{thm:attainability} in the small Hessian perturbation regime
of \cref{sec:pinn-theory}\textup{(i)}.

\begin{theorem}[Finite solution energy]\label{thm:pde-regularity}
Assume \cref{ass:standard}, $K_A<\infty$ and
\cref{eq:hessian-perturbation} with $L_2K_A<1$.
Suppose that $A\in C_b^{1/2,1}$ and $g\in H^2(\mu_T)$.
The unique solution of
\cref{eq:terminal-problem} satisfies
\begin{equation}\label{eq:solution-energy-bound}
 E(u^\star)\le C_{\mathrm{reg}}
       \left(R_{\mathscr H}+\sqrt T+\norm g_{H^2(\mu_T)}\right).
\end{equation}
\end{theorem}
Here $C_{\mathrm{reg}}$ may depend on
$A,q,\mu_0,a,b$ and $L_0,L_1,L_2$, but is independent of
$g$ and $R_{\mathscr H}$. The proof is given in
\cref{app:pde-regularity-proof}.

\section{Numerical comparisons}\label{sec:experiments}
We first tune the backward-training schedule and compare the
first-order and quadratic residuals. We then compare methods across
nonlinear coefficients, study time-grid refinement, and probe
residual stability by terminal translation.
\Cref{app:experimental-settings} gives implementation details,
exact solutions and additional error and residual summaries.
Code and complete experiment records are available at
\url{https://github.com/ZZHPKU/DSRM}.

All problems use $T=1$, $d=10$ and the Brownian occupation measure
$\nu(dt,dx)=dt\,\mathrm N(0,tI)(dx)$.
The Gaussian and sine tests use manufactured solutions of $u_t+F_{c,\rho}(D^2u)+s_{c,\rho}=0$, where
\begin{equation}\label{eq:experimental-nonlinearity}
 F_{c,\rho}(M)=c\operatorname{tr}M+
       \rho\left(\sqrt{\delta^2+|M|_{\mathrm F}^2}-\delta\right),
 \qquad \delta=0.1.
\end{equation}
We set $s_{c,\rho}=-u_t^\star-F_{c,\rho}(D^2u^\star)$, so the exact
solution stays fixed when the coefficients change. For $c>\rho\ge0$,
the ellipticity bounds are $c-\rho$ and $c+\rho$.
The Gaussian comparisons use heat-flow mixtures A, B and C;
\cref{sec:sine-grid} uses sine sums $S_1,S_2,S_3$.
All six solutions are generated once, normalized by
$\norm{D^2u^\star}_\nu=1$, and held fixed; see
\cref{app:gaussian-settings,app:sine-grid-settings} for their construction.

We use the relative full Hessian error and continuous residual norm
\begin{equation*}
 e_H(v)=\norm{D^2(v-u^\star)}_\nu/\norm{D^2u^\star}_\nu,
 \qquad r(v)=\norm{\Res(v)}_\nu.
\end{equation*}
For BR and F1-BR, errors and residuals are evaluated using piecewise
linear interpolation in time of the fitted layers $V_i$.

Each scalar network has three hidden layers of width 128 with tanh
activation. The terminal is imposed exactly by taking $\widehat g=g$
in \cref{eq:hard-terminal,eq:br-hard-terminal}.
The Gaussian and sine comparisons use the budgets in
\cref{tab:experiment-configurations} to keep total residual sample
counts comparable at $N=256$. PINN and G use the same total,
counting one term per occupation point and per time cell along each
path, respectively. BR and F1-BR use about 1.2\% more because of
the extra updates at $i=N-1$.

\begin{table}[htbp]
\centering\small
\setlength{\tabcolsep}{5pt}
\caption{Training configurations. BR and F1-BR use 16,000 updates
at the first backward layer ($i=N-1$) and 4,000 at each remaining layer.}
\label{tab:experiment-configurations}
\begin{tabular}{@{}llll@{}}
\toprule
Method & Representation & Residual terms/update & Total Adam updates\\
\midrule
PINN & One space--time net & 65,536 & 16,000\\
G-C/N & One space--time net & $256N$ (256 paths) & 16,000\\
BR/F1-BR & $N$ spatial nets & 1,024 & $16,000+4,000(N-1)$\\
\bottomrule
\end{tabular}
\end{table}

Checkpoint selection uses the continuous loss for G and PINN and
the local loss for BR and F1-BR. G therefore evaluates $v_t$ during
validation. Training, validation and test samples are independent;
see \cref{app:experiment-protocol} for sample sizes and details.

Gaussian and sine results report means and sample standard deviations
across three exact solutions (one run per solution and configuration), measuring
variation across problems. Translation results use three runs per
fixed problem.

\subsection{Learning-rate tuning for backward regression}
\label{sec:br-experiment}
The first backward fit lacks a warm start and receives extra updates.
Later layers may need corrections of different sizes, motivating
tuning of learning rates and decay times by time region.

We compare uniform and tuned learning-rate schedules at
$(c,\rho)=(0.5,0.15)$ and $N=256$. A BR-C pilot on A selects the tuned
schedule, which is frozen for both endpoints and all three exact solutions.
Both schedules use identical layer budgets;
\cref{app:br-settings} gives the details.

Tuning lowers mean $e_H$ across A/B/C from
6.77\% to 5.46\% for BR-C and from 7.16\% to 5.66\% for BR-N,
with improvements for all three exact solutions at both endpoints
(\cref{tab:br-learning-rate-details}). The tuned runs are reused in
the next two comparisons.

\subsection{First-order and quadratic residuals in backward regression}
\label{sec:first-order-experiment}
We compare F1-BR-C/N and BR-C/N on A/B/C at $(c,\rho)=(0.5,0.15)$
and $N=256$. All use the BR-C schedule from
\cref{sec:br-experiment} and identical layer budgets; F1-BR is not
tuned separately. For C, the residuals differ only by a centered
quadratic increment. For N, the endpoint of diffusion compensation
also differs; see \cref{eq:first-order-difference}.

With C, the mean relative Hessian error decreases from 14.78\% for
F1-BR to 5.46\% for BR; with N, it decreases from 14.58\% to 5.66\%.
BR has smaller Hessian errors for all three exact solutions at both endpoints.
Full results appear in \cref{tab:first-order-results}.

\subsection{Accuracy and training time across nonlinear coefficients}
\label{sec:nonlinear-experiment}
We compare PINN, G-C/N and tuned BR-C/N on A/B/C at the four
coefficient pairs in \cref{tab:nonlinear-hessian-cost}, using $N=256$
for G and BR. Learning-rate schedules remain fixed across coefficient
pairs; BR uses the schedule selected in \cref{sec:br-experiment}.

These pairs compare training within and beyond the structural
regimes in \cref{sec:pinn-theory}.
At $(c,\rho)=(0.5,0.05)$ and $(0.5,0.15)$, the matching
small-perturbation condition $2\sqrt2\rho<1$ and the strict upper
gap both hold. At $(0.75,0.15)$, the ellipticity interval
$[0.60,0.90]$ retains the strict upper gap, but the small-perturbation
test fails for the natural reference $A=cI$, $q=I$;
\cref{eq:experimental-gain-witness} verifies this distinction.
At $(1.25,0.15)$, the interval $[1.10,1.40]$ violates the strict
upper gap while retaining uniform ellipticity.

\begin{table}[tbp]
\centering\small
\setlength{\tabcolsep}{4pt}
\caption{Accuracy and training time across nonlinear coefficients.
Errors $e_H$ (\%) are means $\pm$ sample standard deviations across
A/B/C. Times (minutes) pool all four coefficient pairs and A/B/C
(12 full runs per method); see \cref{app:training-time}.}
\label{tab:nonlinear-hessian-cost}
% Included source: experiments/tables/nonlinear_hessian_cost
% Error statistics aggregated from summary.csv (A/B/C; sample standard deviations).
% Timing values updated from summary.csv (training_seconds; full runs).
\begin{tabular}{@{}lrrrrr@{}}
\toprule
& \multicolumn{4}{c}{$(c,\rho)$} & Time (min)\\
\cmidrule(lr){2-5}
Method & $(0.5,0.05)$ & $(0.5,0.15)$ & $(0.75,0.15)$ & $(1.25,0.15)$ & \\
\midrule
PINN & $5.38\pm0.88$ & $5.40\pm1.12$ & $5.08\pm0.75$ & $5.43\pm0.74$ & $17.118\pm0.007$\\
G-C & $6.12\pm0.97$ & $6.10\pm0.92$ & $5.76\pm0.96$ & $6.00\pm0.79$ & $16.245\pm0.004$\\
G-N & $6.43\pm1.01$ & $6.67\pm0.76$ & $5.84\pm0.94$ & $6.24\pm0.83$ & $16.261\pm0.005$\\
BR-C & $5.41\pm0.73$ & $5.46\pm0.72$ & $5.11\pm0.76$ & $5.11\pm0.94$ & $15.51\pm0.28$\\
BR-N & $5.62\pm0.73$ & $5.66\pm0.73$ & $5.27\pm0.73$ & $5.29\pm0.77$ & $11.91\pm0.41$\\
\bottomrule
\end{tabular}

\end{table}

Tuned BR and PINN have similar Hessian errors across all four
coefficient pairs; G is slightly less accurate. Both BR variants
have lower mean training times, with BR-N fastest.
BR-N mainly saves time by evaluating the driver on the frozen next
layer, avoiding backpropagation through a full current-layer Hessian.

\subsection{Hessian accuracy under time-grid refinement}\label{sec:sine-grid}
We compare observed Hessian-error decay with the $h^{1/2}$ reference
rate in \cref{sec:main-results} for $\vartheta=1$.

To reduce approximation and optimization errors, we use sums
$S_1,S_2,S_3$ of 16 low-frequency sine modes with unit spatial
wavevectors and temporal frequencies $|\omega_j|\le1/4$; see
\cref{app:sine-grid-settings} for their construction.
We compare G-C/N and BR-C/N at $(c,\rho)=(0.5,0.15)$ for
$N=64,128,256,512$, with $h=1/N$. G retains its global schedule;
BR uses the uniform schedule because these solutions are simpler
and vary little in time. With G's global update count and BR's
per-layer counts fixed as in \cref{tab:experiment-configurations},
sampling budgets grow with $N$, so the slopes do not isolate
time-discretization error. The uniform output regularity budget and $O(h^{1/2})$
bounds on the square roots of the population loss criteria required
for this rate are not verified here.

\begin{figure}[tbp]
\centering
\includegraphics[width=\textwidth]{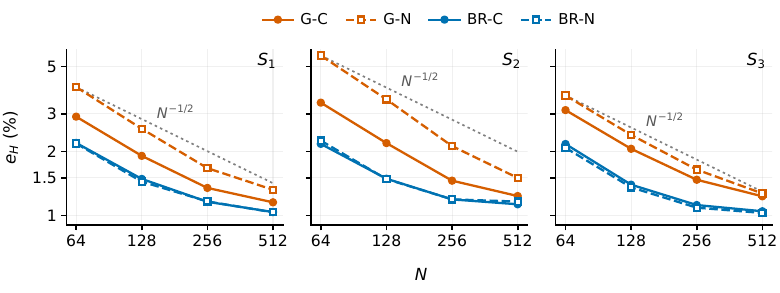}
\caption{Relative Hessian test error $e_H$ on $S_1,S_2,S_3$.
All panels use the same logarithmic axes. Gray dotted guides decay
as $h^{1/2}=N^{-1/2}$ and are anchored to G-N at $N=64$ in each panel.}
\label{fig:sine-grid-hessian}
\end{figure}

Over $N=64$--$256$, the slopes in \cref{fig:sine-grid-hessian}
are 0.54--0.71 for G and 0.43--0.48 for BR.
Over $N=256$--$512$, BR flattens to slopes 0.04--0.17 and Hessian
errors around 1\%, while G has slopes 0.22--0.50.
This suggests that optimization or network approximation errors may
dominate once the grid-dependent contribution becomes small.

\subsection{Residual-to-Hessian control under translation}
\label{sec:linear-translation}
To probe residual stability across terminal data, we translate a
Gaussian terminal while keeping the Brownian sampling law $\nu$ fixed.
We train PINNs on the continuous residual so that time-discretization
error does not enter this comparison:
\begin{equation*}
 u_t+A_{c_\parallel}:D^2u=0,\qquad
 u(1,x)=g_R(x),\qquad
 g_R(x)=\exp(-|x-Re_1|^2/8),
\end{equation*}
where $A_{c_\parallel}=\operatorname{diag}(c_\parallel,\tfrac12,\ldots,\tfrac12)$,
$c_\parallel\in\{0.5,5\}$ and $R\in\{0,10,12\}$.
Both operators are uniformly elliptic; only $c_\parallel=0.5$ satisfies
the strict upper gap relative to $a=I$. We use 4,000 Adam updates
with 65,536 samples per update and importance sampling preserving
$\nu$. \Cref{app:translation-settings} gives the exact solution
$u_{c_\parallel,R}$ and protocol. We use the observed ratios
$K_H(v)=E_H(v)/r(v)$, where
$E_H(v)=\norm{D^2(v-u_{c_\parallel,R})}_\nu$.
These are output ratios, not class suprema.
With $H_\star=\norm{D^2u_{c_\parallel,R}}_\nu$, the identity
$e_H=K_Hr/H_\star$ separates changes in solution scale from changes
in the absolute error-to-residual ratio; \cref{tab:translation-supplement}
reports the exact norms.

\begin{table}[htbp]
\centering\small
\setlength{\tabcolsep}{4pt}
\caption{Terminal translation: means $\pm$ sample standard deviations
over three runs per problem. Mean residual norms $r$ are in physical units;
$K_H$ is computed per run before aggregation.}
\label{tab:translation-stability}
% Included source: experiments/tables/translation_selected
% Aggregated from summary.csv (three repeats; mean residual norms and per-run K_H).
\begin{tabular}{@{}lcccccc@{}}
\toprule
& \multicolumn{3}{c}{$c_\parallel=0.5$} & \multicolumn{3}{c}{$c_\parallel=5$} \\
\cmidrule(lr){2-4}\cmidrule(lr){5-7}
$R$ & $e_H$ (\%) & $r$ & $K_H$ & $e_H$ (\%) & $r$ & $K_H$\\
\midrule
0 & $3.44\pm0.20$ & $5.68\times10^{-3}$ & $1.42\pm0.01$ & $9.19\pm0.25$ & $1.38\times10^{-2}$ & $1.18\pm0.07$\\
10 & $7.26\pm0.25$ & $8.88\times10^{-6}$ & $1.25\pm0.02$ & $79.90\pm0.13$ & $1.56\times10^{-4}$ & $12.21\pm0.17$\\
12 & $17.85\pm0.16$ & $6.93\times10^{-7}$ & $1.20\pm0.01$ & $97.19\pm0.10$ & $9.27\times10^{-6}$ & $51.04\pm0.24$\\
\bottomrule
\end{tabular}

\end{table}

In \cref{tab:translation-stability}, the mean $K_H$ for $c_\parallel=0.5$
remains between 1.20 and 1.42
under translation, even though relative errors rise as $H_\star$
shrinks. For $c_\parallel=5$, $K_H$ rises from 1.18 at $R=0$ to 51.04
at $R=12$, where the mean relative Hessian error is 97.19\%
despite a mean residual of $9.27\times10^{-6}$.

Thus accurate fits in \cref{sec:nonlinear-experiment} do not establish
uniform linear residual control. The growth of $K_H$ is compatible
with $E_H\le Cr^\vartheta$, $\vartheta<1$, on the regular classes
of criterion~(iii) in \cref{sec:pinn-theory}.

\section{Conclusions}
Under a uniform output regularity budget, the quadratic residual bounds the
continuous PDE loss by the G loss or BR excess loss,
up to discretization and terminal terms. Residual stability converts
this estimate into value, gradient and Hessian error bounds.
The neural construction shows that this budget and the required
loss accuracy can be attained simultaneously.

The main open problem is a discrete analogue of
\cref{eq:pinn-stability} that controls value, gradient and Hessian
errors at grid points without an additional output regularity budget.
Such an estimate would control the fitted layers directly, without
requiring a regular space--time realization. Brownian sampling is
a natural starting point. Further questions concern error bounds
for discretized SDE sampling, quantitative network complexity and
training guarantees, including uniform regularity of trained outputs.

\FloatBarrier
\raggedbottom
\appendix

\section{Experimental details and additional results}\label{app:experimental-settings}
\subsection{Implementation and evaluation protocol}\label{app:experiment-protocol}
All MLPs have linear outputs. Random initialization draws the weights
and biases of each affine layer with $n$ inputs independently from
$U[-n^{-1/2},n^{-1/2}]$. Adam \cite{kingma-ba-2015} uses
$(\beta_1,\beta_2)=(0.9,0.999)$ and $\epsilon=10^{-8}$, without weight
decay or gradient clipping. G and Gaussian PINN reduce the learning
rate from $10^{-3}$ to $3\times10^{-4}$ after 12,000 updates, retaining
Adam's moments. BR and F1-BR initialization and schedules are given in
\cref{app:br-settings}.

Training uses float32 with mixed precision and TF32 both disabled;
final evaluation uses float64 on the CPU. Exact-solution derivatives
are evaluated analytically and network derivatives by explicit
chain-rule propagation. Whenever second spatial derivatives are needed,
we compute full Hessians, with two exceptions: (i) BR-N evaluates the
current-layer quadratic correction directly as a directional second
derivative; (ii) the translated linear PINN residuals use only diagonal
Hessian entries.

Each training update uses fresh samples. For the Gaussian and sine
problems, G subsamples Brownian paths from a 512-step master grid,
while PINN draws IID occupation points. At layer $i$, BR and F1-BR draw
$X_i\sim\mathrm N(0,t_iI)$ and independent increments
$\Delta W_i\sim\mathrm N(0,hI)$; local validation uses the same law.
Training minimizes sample estimates of $\Loss_{h,\mathrm{int}}^\varpi$ for G,
$h^{-1}\ell_i^\varpi$ for BR, and $r^2$ for PINN.
F1-BR uses the same local scaling with $\widehat\ell_i^\varpi$
from \cref{sec:residual-formulation}.

For the Gaussian and sine comparisons, validation uses a fixed bank
of 16,384 samples independent of training. G and PINN save and
validate checkpoints every 500 Adam updates using $\Loss(v)=r(v)^2$
on IID occupation points. BR and F1-BR validate every 100 updates
per layer using $h^{-2}\ell_i^\varpi(V_i;V_{i+1})$ and
$h^{-2}\widehat\ell_i^\varpi(V_i;V_{i+1})$, respectively.
The iterate with the smallest validation loss is selected.
Candidates include initialization and the final update, with ties
resolved in favor of the earlier iterate.
Final errors use 262,144 IID test points, independent of training
and validation and unused in tuning or selection.

In addition to $e_H$, the tables below report relative value and
gradient errors
$e_k(v)=\norm{D^k(v-u^\star)}_\nu/\norm{D^ku^\star}_\nu$, $k=0,1$.
The denominator energies use analytic Gaussian integration in space
and 256-point Gauss--Legendre quadrature in time.
The repository linked in \cref{sec:experiments} contains reproduction
scripts, per-run results, configurations, seeds, validation histories
and pilot records.

\subsection{Backward-regression initialization and schedules}\label{app:br-settings}
For BR and F1-BR, the first spatial network is initialized from the
$t_{N-1}$ slice of a randomly initialized space--time network.
Adam's moments are reset at each layer.

The uniform rates use the scale identified in preliminary pilots.
The learning-rate search uses BR-C on Gaussian exact solution A at
$(c,\rho)=(0.5,0.15)$ and $N=256$. A pilot backward solution supplies
fixed next-layer targets and local initial weights.
The five groups are layer 255, $192$--$254$, $128$--$191$,
$64$--$127$ and $0$--$63$, represented by 255, $(240,208)$,
$(176,144)$, $(112,80)$ and $(48,16)$, respectively.
Five initial rates are screened in each group, followed by tuning
of the decay times. The groups $64$--$127$ and
$128$--$191$ select the same schedule and are combined in
\cref{tab:br-schedules}.

Full runs restart from the saved initialization.
The selected schedule is frozen for both endpoints,
B/C, the other coefficient pairs and F1-BR.

\begin{table}[htbp]
\centering\small
\setlength{\tabcolsep}{4pt}
\caption{Frozen BR learning-rate schedules at $N=256$, in execution
order within each layer. Uniform means a common decay rule for
ordinary layers. For other grids, its first layer is $N-1$ and
ordinary layers are $0,\ldots,N-2$.}
\label{tab:br-schedules}
\begin{tabular}{@{}llll@{}}
\toprule
Schedule & Layer $i$ & Updates per stage & Learning rates\\
\midrule
Uniform & 255 & 16,000 & $10^{-3}$\\
 & $0$--$254$ & 1,000 each & $10^{-4},\ 5\cdot10^{-5},\ 2.5\cdot10^{-5},\ 1.25\cdot10^{-5}$\\
Tuned & 255 & 8,000 / 4,000 / 4,000 & $4\cdot10^{-3},\ 2\cdot10^{-3},\ 10^{-3}$\\
 & $192$--$254$ & 2,000 / 1,000 / 1,000 & $5\cdot10^{-5},\ 2.5\cdot10^{-5},\ 1.25\cdot10^{-5}$\\
 & $64$--$191$ & 2,000 / 1,000 / 1,000 & $2.5\cdot10^{-5},\ 1.25\cdot10^{-5},\ 6.25\cdot10^{-6}$\\
 & $0$--$63$ & 4,000 & $1.5625\cdot10^{-5}$\\
\bottomrule
\end{tabular}
\end{table}

\begin{table}[htbp]
\centering\small
\setlength{\tabcolsep}{6pt}
\caption{Relative Hessian errors $e_H$ (\%) for the learning-rate
comparison in \cref{sec:br-experiment}.}
\label{tab:br-learning-rate-details}
% Included source: experiments/tables/br_learning_rate_details
% Relative Hessian errors (%) from summary.csv; one run per exact solution and schedule.
\begin{tabular}{@{}lrrrrrr@{}}
\toprule
& \multicolumn{2}{c}{A} & \multicolumn{2}{c}{B} & \multicolumn{2}{c}{C}\\
\cmidrule(lr){2-3}\cmidrule(lr){4-5}\cmidrule(lr){6-7}
Method & Uniform & Tuned & Uniform & Tuned & Uniform & Tuned\\
\midrule
BR-C & 7.710 & 6.074 & 6.709 & 5.653 & 5.889 & 4.664\\
BR-N & 8.139 & 6.426 & 7.107 & 5.574 & 6.248 & 4.982\\
\bottomrule
\end{tabular}

\end{table}

\subsection{Gaussian exact solutions and comparisons}\label{app:gaussian-settings}
For each of A, B and C, independently draw twelve centers
$z_j\sim\mathrm N(0,0.75^2I_{10})$, widths
$\sigma_j\sim U[1.4,1.8]$, and signs
$a_j\in\{-1,1\}/\sqrt{12}$ with equal probabilities. Set
$\tau_j(t)=\sigma_j^2+2(1-t)$ and define
\begin{equation*}
 u^\star(t,x)=C\sum_{j=1}^{12}a_j
       \left(\frac{\sigma_j^2}{\tau_j(t)}\right)^5
       e^{-|x-z_j|^2/(2\tau_j(t))},\qquad g=u^\star(1,\cdot).
\end{equation*}
Choose $C>0$ by the Hessian normalization in \cref{sec:experiments}.

For the first-order comparison in \cref{sec:first-order-experiment},
\cref{tab:first-order-results} reports the errors and residual losses.
F1-BR attains mean value and gradient errors of about 1--2\% and 5\%,
respectively. BR has smaller continuous residual losses for all three
exact solutions at both endpoints.

\begin{table}[htbp]
\centering\small
\setlength{\tabcolsep}{4pt}
\caption{First-order and quadratic residuals in backward regression
at $(c,\rho)=(0.5,0.15)$ and $N=256$.
Entries are means $\pm$ sample standard deviations across A/B/C
under the BR-tuned schedule and equal layer budgets.
The BR rows reuse the tuned runs from \cref{sec:br-experiment}.}
\label{tab:first-order-results}
% Included source: experiments/tables/first_order_results
% Aggregated from summary.csv (A/B/C; shared tuned schedule; sample standard deviations).
\begin{tabular}{@{}lrrrr@{}}
\toprule
Method & $e_0$ (\%) & $e_1$ (\%) & $e_H$ (\%) & $10^3r^2$\\
\midrule
BR-C & $0.814\pm0.199$ & $2.167\pm0.216$ & $5.464\pm0.724$ & $1.506\pm0.346$\\
F1-BR-C & $1.497\pm0.502$ & $4.808\pm0.608$ & $14.782\pm2.604$ & $10.251\pm2.835$\\
\midrule
BR-N & $1.052\pm0.187$ & $2.390\pm0.256$ & $5.660\pm0.726$ & $1.646\pm0.361$\\
F1-BR-N & $1.303\pm0.441$ & $4.594\pm0.645$ & $14.579\pm2.717$ & $9.883\pm2.774$\\
\bottomrule
\end{tabular}

\end{table}

For the coefficient comparison in \cref{sec:nonlinear-experiment},
we verify the failed perturbation test at $(c,\rho)=(0.75,0.15)$.
Take $A=cI$, $q=I$, and set
$\gamma_{10}=\mathrm N(0,I_{10})$ and
$\mathcal A_c=-c\Delta+\tfrac12x\cdot D$.
The point-start formula in
\cite[Proposition~4.4 and Corollary~4.5]{long-zhao-2026} gives
\begin{equation*}
 K_{cI}\ge
 \frac{\norm{D^2\psi}_{L^2(\gamma_{10})}}
      {\norm{(\mathcal A_c-\tfrac12)\psi}_{L^2(\gamma_{10})}}.
\end{equation*}
The bound also holds for polynomial $\psi$: spatial cutoff
approximations converge in both norms in the ratio, since all
Gaussian polynomial moments are finite.
For $c=3/4$, take $\psi(x)=p(|x|^2)$ with
\begin{equation*}
 p(z)=z^4-32z^3+700z^2+8576z-342880.
\end{equation*}
Using $\E_{\gamma_{10}}|x|^{2j}=\prod_{k=0}^{j-1}(10+2k)$ yields
\begin{equation}\label{eq:experimental-gain-witness}
 K_{(3/4)I}^2\ge\frac{19296944}{391099},\qquad
 (0.15K_{(3/4)I})^2\ge\frac{10854531}{9777475}>1.
\end{equation}
The global Hessian Lipschitz constant of $F_{c,\rho}-c\operatorname{tr}$
is exactly $\rho$. Thus this reference fails the perturbation criterion;
the calculation does not exclude alternative reference operators.

\Cref{tab:gaussian-supplement} gives the lower-order errors and
continuous residual losses for this coefficient comparison. The size of
$r^2$ also depends on the operator scale, so its increase with $c$
need not imply a larger Hessian error.

\begin{table}[htbp]
\centering\small
\setlength{\tabcolsep}{1.5pt}
\caption{Supplementary Gaussian errors: means $\pm$ sample standard
deviations across A/B/C. BR uses the tuned schedule except for the
uniform controls from \cref{sec:br-experiment} marked \textsuperscript{u}.
$e_0$ and $e_1$ are relative errors in percent.}
\label{tab:gaussian-supplement}
% Included source: experiments/tables/gaussian_supplement
% Aggregated from summary.csv (A/B/C; sample standard deviations).
\begin{tabular}{@{}lrrrrrr@{}}
\toprule
 & \multicolumn{3}{c}{$(c,\rho)=(0.5,0.05)$} & \multicolumn{3}{c}{$(c,\rho)=(0.5,0.15)$}\\
\cmidrule(lr){2-4}\cmidrule(lr){5-7}
Method & $e_0$ (\%) & $e_1$ (\%) & $10^3r^2$ & $e_0$ (\%) & $e_1$ (\%) & $10^3r^2$\\
\midrule
PINN & $0.792\pm0.259$ & $2.184\pm0.342$ & $1.402\pm0.439$ & $0.800\pm0.381$ & $2.181\pm0.437$ & $1.400\pm0.608$\\
G-C & $1.054\pm0.327$ & $2.660\pm0.288$ & $1.945\pm0.542$ & $1.077\pm0.291$ & $2.732\pm0.256$ & $1.900\pm0.528$\\
G-N & $1.287\pm0.460$ & $2.982\pm0.485$ & $2.159\pm0.744$ & $1.407\pm0.632$ & $3.100\pm0.368$ & $2.203\pm0.588$\\
BR-C & $0.761\pm0.192$ & $2.138\pm0.220$ & $1.506\pm0.335$ & $0.814\pm0.199$ & $2.167\pm0.216$ & $1.506\pm0.346$\\
BR-N & $0.998\pm0.169$ & $2.355\pm0.250$ & $1.658\pm0.354$ & $1.052\pm0.187$ & $2.390\pm0.256$ & $1.646\pm0.361$\\
\addlinespace[2pt]
BR-C\textsuperscript{u} & \textemdash & \textemdash & \textemdash & $0.968\pm0.276$ & $2.729\pm0.322$ & $2.242\pm0.602$\\
BR-N\textsuperscript{u} & \textemdash & \textemdash & \textemdash & $1.193\pm0.243$ & $2.981\pm0.339$ & $2.501\pm0.632$\\
\midrule
 & \multicolumn{3}{c}{$(c,\rho)=(0.75,0.15)$} & \multicolumn{3}{c}{$(c,\rho)=(1.25,0.15)$}\\
\cmidrule(lr){2-4}\cmidrule(lr){5-7}
Method & $e_0$ (\%) & $e_1$ (\%) & $10^3r^2$ & $e_0$ (\%) & $e_1$ (\%) & $10^3r^2$\\
\midrule
PINN & $0.838\pm0.401$ & $2.152\pm0.342$ & $1.742\pm0.582$ & $1.486\pm0.472$ & $3.274\pm0.736$ & $2.588\pm0.768$\\
G-C & $1.027\pm0.243$ & $2.587\pm0.343$ & $2.488\pm0.781$ & $1.559\pm0.636$ & $3.834\pm0.918$ & $3.729\pm0.938$\\
G-N & $1.105\pm0.296$ & $2.698\pm0.346$ & $2.595\pm0.762$ & $1.885\pm0.354$ & $4.050\pm0.585$ & $3.967\pm1.000$\\
BR-C & $0.676\pm0.220$ & $1.977\pm0.228$ & $2.037\pm0.536$ & $0.896\pm0.483$ & $2.106\pm0.350$ & $3.406\pm0.993$\\
BR-N & $0.800\pm0.182$ & $2.130\pm0.242$ & $2.197\pm0.537$ & $0.985\pm0.290$ & $2.237\pm0.297$ & $3.637\pm1.016$\\
\bottomrule
\end{tabular}

\end{table}

\subsection{Sine exact solutions and grid comparisons}\label{app:sine-grid-settings}
For each of $S_1,S_2,S_3$, draw independent coefficients
$a_j\sim\mathrm N(0,1)$, directions $w_j$ by normalizing independent
$\mathrm N(0,I_{10})$ vectors, frequencies
$\omega_j\sim U[-1/4,1/4]$, and phases $b_j\sim U[-\pi,\pi]$.
Set
\begin{equation*}
 u^\star(t,x)=\frac C4\sum_{j=1}^{16}a_j
       \sin(w_j\cdot x+\omega_jt+b_j),\qquad g=u^\star(1,\cdot),
\end{equation*}
where $C>0$ is chosen by the same Hessian normalization.

The empirical slopes in \cref{sec:sine-grid} are
$p=\log(e_{N_a}/e_{N_b})/\log(N_b/N_a)$, with
$e_N=e_H(v_N)$ evaluated on the selected trained output.
Supplementary value and gradient errors appear in
\cref{tab:sine-supplement}.

\begin{table}[htbp]
\centering\small
\setlength{\tabcolsep}{4pt}
\caption{Supplementary sine errors (\%). Each cell lists $e_0$ above
$e_1$, as mean $\pm$ sample standard deviation across $S_1,S_2,S_3$.}
\label{tab:sine-supplement}
% Included source: experiments/tables/sine_supplement
% Aggregated from summary.csv (S1/S2/S3; sample standard deviations).
\begin{tabular}{@{}lcccc@{}}
\toprule
Method & $N=64$ & $N=128$ & $N=256$ & $N=512$\\
\midrule
G-C & \begin{tabular}{@{}r@{}}$0.536\pm0.114$\\[-1pt]$1.449\pm0.171$\end{tabular} & \begin{tabular}{@{}r@{}}$0.331\pm0.060$\\[-1pt]$0.900\pm0.125$\end{tabular} & \begin{tabular}{@{}r@{}}$0.237\pm0.014$\\[-1pt]$0.620\pm0.121$\end{tabular} & \begin{tabular}{@{}r@{}}$0.196\pm0.037$\\[-1pt]$0.534\pm0.069$\end{tabular}\\
\addlinespace[3pt]
G-N & \begin{tabular}{@{}r@{}}$0.678\pm0.164$\\[-1pt]$1.883\pm0.017$\end{tabular} & \begin{tabular}{@{}r@{}}$0.420\pm0.105$\\[-1pt]$1.159\pm0.031$\end{tabular} & \begin{tabular}{@{}r@{}}$0.276\pm0.043$\\[-1pt]$0.758\pm0.058$\end{tabular} & \begin{tabular}{@{}r@{}}$0.233\pm0.017$\\[-1pt]$0.584\pm0.107$\end{tabular}\\
\addlinespace[3pt]
BR-C & \begin{tabular}{@{}r@{}}$0.379\pm0.068$\\[-1pt]$1.102\pm0.080$\end{tabular} & \begin{tabular}{@{}r@{}}$0.207\pm0.046$\\[-1pt]$0.629\pm0.039$\end{tabular} & \begin{tabular}{@{}r@{}}$0.127\pm0.011$\\[-1pt]$0.417\pm0.064$\end{tabular} & \begin{tabular}{@{}r@{}}$0.098\pm0.012$\\[-1pt]$0.343\pm0.055$\end{tabular}\\
\addlinespace[3pt]
BR-N & \begin{tabular}{@{}r@{}}$0.381\pm0.076$\\[-1pt]$1.067\pm0.028$\end{tabular} & \begin{tabular}{@{}r@{}}$0.204\pm0.034$\\[-1pt]$0.605\pm0.036$\end{tabular} & \begin{tabular}{@{}r@{}}$0.127\pm0.013$\\[-1pt]$0.413\pm0.057$\end{tabular} & \begin{tabular}{@{}r@{}}$0.098\pm0.013$\\[-1pt]$0.342\pm0.050$\end{tabular}\\
\bottomrule
\end{tabular}

\end{table}

\subsection{Translated linear problems}\label{app:translation-settings}
For the problem in \cref{sec:linear-translation}, set
$s_\parallel=4+2c_\parallel(1-t)$ and $s_\perp=5-t$. The exact solution is
\begin{equation*}
 u_{c_\parallel,R}(t,x)=
 \left(\frac4{s_\parallel}\right)^{1/2}
 \left(\frac4{s_\perp}\right)^{9/2}
 \exp\left(-\frac{(x_1-R)^2}{2s_\parallel}
           -\frac{\sum_{j=2}^{10}x_j^2}{2s_\perp}\right).
\end{equation*}
Set $H_{g,R}=\norm{D^2g_R}_\nu$ with $g_R$ treated as time-independent.
To normalize the optimization scale and center the inputs, use
$\mathcal N_\theta(t,x)
 =H_{g,R}\widetilde{\mathcal N}_\theta(t,x-Re_1)$
in \cref{eq:hard-terminal}.

For training, validation and testing, draw $t\sim U[0,1]$ and
$x\mid t\sim\pi_{t,R}$. Writing $p_t$ for the density of
$\mathrm N(0,tI)$, the proposal
$\pi_{t,R}$ assigns probability $1/2$ to $p_t$ and $1/6$ to each
normalized density proportional to
$p_tg_R^2$, $p_tu_{0.5,R}^2$ and $p_tu_{5,R}^2$.
Each tilted component has independent Gaussian coordinates with
means and variances
\begin{equation*}
 \widetilde m_i=\frac{2t m_i}{s_i+2t},\qquad
 \widetilde\sigma_i^2=\frac{ts_i}{s_i+2t},\qquad m=Re_1.
\end{equation*}
Use $s_i=4$ for $g_R$; for each solution profile, use
$s_1=s_\parallel$ and $s_i=s_\perp$ for $i\ge2$.
Both operators share the same proposal at a fixed $R$.
For any squared diagnostic $f$, estimate its occupation integral by
\begin{equation*}
 \frac1B\sum_{b=1}^B
 \frac{p_{t_b}(x_b)}{\pi_{t_b,R}(x_b)}f(t_b,x_b).
\end{equation*}
The weights are bounded by 2 and are used without self-normalization.
Training minimizes this estimate of $r(v_\theta)^2/H_{g,R}^2$;
reported errors and residuals use the physical output $v_\theta$.

Training uses learning rate $10^{-3}$ under the budget in
\cref{sec:linear-translation}.
Validation uses the weighted estimate of $r(v)^2$, with sample sizes
and checkpoint rules as in \cref{app:experiment-protocol}.
For each fixed $(R,c_\parallel)$, the three runs use the same exact solution,
with independent network initializations and independently generated
training, validation and test samples.

\begin{table}[htbp]
\centering\small
\setlength{\tabcolsep}{3pt}
\caption{Supplementary translation results. Relative errors $e_0$
and $e_1$ are percentages, averaged over three runs
with sample standard deviations.
$H_\star=\norm{D^2u_{c_\parallel,R}}_\nu$ is the exact Hessian norm.}
\label{tab:translation-supplement}
% Included source: experiments/tables/translation_supplement
% Error statistics aggregated from summary.csv; exact Hessian norms unchanged.
\begin{tabular}{@{}lrrrrrr@{}}
\toprule
& \multicolumn{3}{c}{$c_\parallel=0.5$} & \multicolumn{3}{c}{$c_\parallel=5$} \\
\cmidrule(lr){2-4}\cmidrule(lr){5-7}
$R$ & $e_0$ (\%) & $e_1$ (\%) & $H_\star$ & $e_0$ (\%) & $e_1$ (\%) & $H_\star$\\
\midrule
0 & $0.309\pm0.014$ & $1.832\pm0.134$ & $2.34\times10^{-1}$ & $3.995\pm0.771$ & $6.057\pm0.221$ & $1.77\times10^{-1}$\\
10 & $3.045\pm0.430$ & $4.383\pm0.153$ & $1.53\times10^{-4}$ & $67.982\pm0.367$ & $72.148\pm0.255$ & $2.39\times10^{-3}$\\
12 & $9.785\pm0.227$ & $13.069\pm0.117$ & $4.67\times10^{-6}$ & $94.728\pm0.250$ & $95.850\pm0.159$ & $4.87\times10^{-4}$\\
\bottomrule
\end{tabular}

\end{table}

\subsection{Training time}\label{app:training-time}
Training time covers all updates in the prescribed budget on an
NVIDIA RTX 4070 SUPER, including updates after the selected checkpoint.
It includes sampling, evaluation of the known terminal and source,
network differentiation, backpropagation and Adam.
Compilation, validation, checkpoint writes, final evaluation and
learning-rate pilots are excluded.
The timed runs use the same batch sizes, precision and derivative
implementations as the error comparisons.

Training times are reported in minutes as means $\pm$ sample standard
deviations. \Cref{tab:nonlinear-hessian-cost} pools 12 runs per method:
A/B/C at each of the four coefficient pairs, with the tuned schedule
for BR.

% Included source: experiments/tables/training_time_summary
% Generated by experiments/write_experiment_presentation.py.
% Timing values updated from summary.csv (training_seconds; full runs).
For Gaussian F1-BR at $(c,\rho)=(0.5,0.15)$ and $N=256$ with the frozen
tuned schedule, the training times across A/B/C (three runs per endpoint)
are $15.55\pm0.36$ min (C) and $10.69\pm0.51$ min (N).
The matched BR runs take $15.50\pm0.32$ min (C) and
$11.99\pm0.50$ min (N). Thus F1-BR-C takes 0.4\% longer and F1-BR-N
10.9\% less time than BR in \cref{sec:first-order-experiment}.

\Cref{tab:training-times} pools three runs per method and grid size,
one for each of $S_1,S_2,S_3$, with the uniform schedule for BR.

\begin{table}[htbp]
\centering\small
\setlength{\tabcolsep}{4pt}
\caption{Sine training times on an RTX 4070 SUPER, in minutes.
Entries are means $\pm$ sample standard deviations across $S_1,S_2,S_3$.
Budgets are those in \cref{tab:experiment-configurations}.}
\label{tab:training-times}
% Included source: experiments/tables/training_times
% Generated by experiments/write_experiment_presentation.py.
% Timing values updated from summary.csv (training_seconds; full runs).
\begin{tabular}{@{}lrrrr@{}}
\toprule
Method & \multicolumn{4}{c}{Sine}\\
\cmidrule(lr){2-5}
& $N=64$ & $N=128$ & $N=256$ & $N=512$\\
\midrule
G-C & $4.07\pm0.26$ & $7.57\pm0.06$ & $15.42\pm0.71$ & $29.97\pm0.41$\\
G-N & $4.03\pm0.20$ & $7.71\pm0.21$ & $15.24\pm0.41$ & $30.35\pm0.94$\\
BR-C & $3.829\pm0.004$ & $7.565\pm0.007$ & $15.009\pm0.009$ & $30.33\pm0.05$\\
BR-N & $2.9427\pm0.0009$ & $5.775\pm0.006$ & $11.512\pm0.008$ & $23.02\pm0.07$\\
\bottomrule
\end{tabular}

\end{table}

The translated linear PINN takes $2.20\pm0.04$ min, pooling
$R\in\{0,10,12\}$, $c_\parallel\in\{0.5,5\}$ and three repeats per problem
(18 runs) under the budget in \cref{sec:linear-translation}.

\FloatBarrier
\section{Remaining proofs and auxiliary lemmas}\label{app:remaining-proofs}
\subsection{Finite solution energy}\label{app:pde-regularity-proof}
\begin{proof}[Proof of \cref{thm:pde-regularity}]
\emph{Linear estimates.}
For $K_A<\infty$, let $\mathscr L_A(f,\tau)$ solve
$P_Az=f$, $z(T)=\tau$.
Then \cite[Theorem~3.6]{long-zhao-2026} gives
\begin{equation}\label{eq:linear-data-bound}
 \norm{\mathscr L_A(f,\tau)}_{\mathbb X_\nu}
 \le C_{\mathrm{lin}}(\norm f_\nu+\norm\tau_{H^1(\mu_T)}).
\end{equation}
Here $C_{\mathrm{lin}}$ may depend on $A$, the initial law $\mu_0$
and the sampling coefficients $a,b$, but is independent of $f,\tau$.
Define $C_0(\kappa)$ and $C_1(\kappa)$ as in
\cref{eq:weighted-gains}, replacing the numerator by
$\norm w_{\kappa,\nu}$ and $\norm{Dw}_{\kappa,\nu}$, respectively.
By \cite[Proposition~3.7]{long-zhao-2026},
$C_0(\kappa)=O(\kappa^{-1})$, $C_1(\kappa)=O(\kappa^{-1/2})$
and $C_2(\kappa)\to K_A$. When $L_2K_A<1$, fix $\kappa$ large enough that
\begin{equation}
 \varrho_\kappa
 =L_0C_0(\kappa)+L_1C_1(\kappa)+L_2C_2(\kappa)<1.
\end{equation}

\emph{Source iteration.}
By \cite[Proposition~5.1]{long-zhao-2026}, the source iteration
\begin{equation}
 s_0=0,\qquad u_n=\mathscr L_A(s_n,g),\qquad s_{n+1}=N_A[J^2u_n]
\end{equation}
satisfies $u_n\to u^\star$ in $\mathbb X_\nu$.
The source map is a contraction in $\norm{\cdot}_{\kappa,\nu}$
with factor $\varrho_\kappa<1$. Since $s_0=0$,
\begin{equation*}
 \sup_n\norm{s_n}_{\kappa,\nu}
 \le\frac{\norm{s_1}_{\kappa,\nu}}{1-\varrho_\kappa}
 \le C_{\mathrm{reg}}(R_{\mathscr H}+\norm g_{H^1(\mu_T)}).
\end{equation*}
Here $s_1=N_A[J^2\mathscr L_A(0,g)]$, so the last bound follows from
\cref{eq:hessian-perturbation,eq:linear-data-bound}.
Using $\norm f_\nu\le e^{\kappa T}\norm f_{\kappa,\nu}$ in
\cref{eq:linear-data-bound} then yields
\begin{equation}\label{eq:source-iterate-bound}
 \sup_n\norm{u_n}_{\mathbb X_\nu}
       \le C_{\mathrm{reg}}(R_{\mathscr H}+\norm g_{H^1(\mu_T)}).
\end{equation}

\emph{Differentiating the iterates.}
We isolate the response to $Ds_n$ with zero terminal data to retain
the contraction factor $\varrho_\kappa$ when estimating $Ds_{n+1}$.
Write $S_Af:=\mathscr L_A(f,0)$ for the zero-terminal solver.
Suppose $Ds_n\in L^2(\nu)$. Spatial differentiation gives
$Du_n\in\mathbb X_\nu$ and, for $1\le j\le d$,
\begin{equation}\label{eq:source-gradient-decomposition}
 D_ju_n=S_AD_js_n+k_{n,j},\quad
 k_{n,j}=\mathscr L_A((D_jA):D^2u_n,D_jg).
\end{equation}
For $k_n=(k_{n,1},\ldots,k_{n,d})$, \cref{eq:linear-data-bound} yields
\begin{equation}\label{eq:source-gradient-remainder}
 \norm{k_n}_{\mathbb X_\nu}
 \le C_{\mathrm{lin}}(\norm{DA}_\infty\norm{u_n}_{\mathbb X_\nu}
                              +\norm g_{H^2(\mu_T)}).
\end{equation}
For smooth, compactly supported $s_n,g$, the difference-quotient
argument for \cref{eq:matching-gradient-equation} gives
\cref{eq:source-gradient-decomposition}, since $A$ is spatially Lipschitz.
For general data, cutoff and mollification approximate $s_n$
together with $Ds_n$ in $L^2(\nu)$ and $g$ in $H^2(\mu_T)$,
because $d\mu_t/dx$ is bounded above and away from zero on compact
subsets of $(0,T]\times\R^d$.
Applying \cref{eq:linear-data-bound} to differences of the linear
and differentiated equations gives convergence of the approximating
solutions and their spatial gradients in $\mathbb X_\nu$.
This proves \cref{eq:source-gradient-decomposition,eq:source-gradient-remainder}
for the stated data.

Since $u_n,Du_n\in\mathbb X_\nu$ and the sampling density is locally
bounded away from zero, the spatial derivatives of $u_n$ through
order three lie in $L^2_{\mathrm{loc}}(dt\,dx)$ on $(0,T)\times\R^d$.
For $s_{n+1}=N_A(t,x,J^2u_n)$, split spatial difference quotients
into jet increments at fixed $x$ and $x$-increments at fixed jet.
Bound these by \cref{eq:hessian-perturbation} and by
\cref{ass:standard} with $\norm{DA}_\infty$, respectively.
The local $L^2$ difference-quotient criterion then gives, almost everywhere,
\begin{equation}\label{eq:nonlinear-spatial-chain}
 |D_js_{n+1}|
 \le L_0|D_ju_n|+L_1|DD_ju_n|+L_2|\mathcal H_qD_ju_n|
       +C(1+\norm{DA}_\infty)(1+|J^2u_n|).
\end{equation}
Thus $Ds_{n+1}\in L^2(\nu)$, and induction from $Ds_0=0$ gives
$Du_n\in\mathbb X_\nu$ and $Ds_n\in L^2(\nu)$.

\emph{Bounding $Ds_n$ and $E(u_n)$.}
The constants $C_0(\kappa),C_1(\kappa),C_2(\kappa)$ also apply to
vector sources, by summing the squared scalar inequalities over components.
Substitute $Du_n=S_ADs_n+k_n$ from
\cref{eq:source-gradient-decomposition} into \cref{eq:nonlinear-spatial-chain}.
After taking the $\norm{\cdot}_{\kappa,\nu}$ norm, the terms involving
$S_ADs_n$, $DS_ADs_n$ and $\mathcal H_qS_ADs_n$ contribute at most
$\varrho_\kappa\norm{Ds_n}_{\kappa,\nu}$.
Using \cref{eq:source-gradient-remainder} for $k_n$ and
\cref{eq:source-iterate-bound} for $u_n$ gives
\begin{equation}
 \norm{Ds_{n+1}}_{\kappa,\nu}
 \le\varrho_\kappa\norm{Ds_n}_{\kappa,\nu}
 +C_{\mathrm{reg}}\left(R_{\mathscr H}+\sqrt T+\norm g_{H^2(\mu_T)}\right).
\end{equation}
Iteration from $Ds_0=0$ and $\varrho_\kappa<1$ bounds
$\norm{Ds_n}_{\kappa,\nu}$ uniformly in $n$. Then
\cref{eq:source-gradient-decomposition,eq:linear-data-bound,eq:source-gradient-remainder}
give
\begin{equation*}
 \sup_n E(u_n)\le C_{\mathrm{reg}}
 \left(R_{\mathscr H}+\sqrt T+\norm g_{H^2(\mu_T)}\right).
\end{equation*}
The convergence $u_n\to u^\star$ in $\mathbb X_\nu$ implies
$Du_n\to Du^\star$ in $L^2(\nu)$, so any weakly convergent
subsequence of $Du_n$ in $\mathbb X_\nu$ has limit $Du^\star$.
Weak lower semicontinuity gives \cref{eq:solution-energy-bound}.
\end{proof}

\subsection{Smooth approximation}\label{app:energy-estimates}
\begin{lemma}\label{lem:energy-density}
Under \cref{ass:standard}, every $f$ with $E(f)<\infty$ has approximations
$w_m\in C_c^\infty(\R^{d+1})$ with $E(w_m-f)\to0$.
If $f(T)=0$, they can satisfy $w_m(T)=0$.
If $J^2f$ is locally continuous near $t=0$ and
$J^2f(0)\in L^2(\mu_0)$, then $d_h(w_m,f)\to0$ for each fixed mesh,
with $d_h$ as in \cref{eq:grid-density-topology}.
If also $\norm f_\gamma<\infty$, the same sequence satisfies
\begin{equation}\label{eq:holder-smooth-density}
 \sup_m\norm{w_m}_\gamma\le C_\gamma\norm f_\gamma.
\end{equation}
\end{lemma}
\begin{proof}
For $S_\delta$ from \cref{eq:forward-smoothing},
Jensen's inequality and \cref{eq:local-kernel-lower-bound} give,
for $F\in L^2(\nu)$,
\begin{equation}\label{eq:forward-smoothing-marginal}
 \E|S_\delta F(t,X_t)|^2
 \le\frac C\delta\int_{\delta/2}^{\delta}
                  \E|F(t+r,X_{t+r})|^2dr.
\end{equation}
Integration in $t$ gives
$\norm{S_\delta F}_{\nu,(0,T')}\le C\norm F_\nu$
for $T'<T$ and $0<\delta<T-T'$.
Since $S_\delta F\to F$ uniformly for $F\in C_c^\infty$, density in
$L^2(\nu)$ gives $\norm{S_\delta F-F}_{\nu,(0,T')}\to0$
for every $F\in L^2(\nu)$.
Equation~\eqref{eq:forward-smoothing-marginal} also holds with
$\rho$ or $\psi$ replaced by $|\rho'|$ or $|D_k\psi|$, respectively.

To construct $w_{R,\delta}$, take a fixed spatial cutoff
$\zeta\in C_c^\infty$ equal to one near zero and set
$f_R=\zeta(\cdot/R)f$, $R\ge1$.
We use $S_\delta f_R$ near $0$, where $\mu_0$ may be singular,
and ordinary convolution after extending $f_R$ across $T$ by
\begin{equation*}
 \widetilde f_R(T+s,x)=3f_R(T-s,x)-2f_R(T-2s,x),\quad 0<s<T/6.
\end{equation*}
When $\norm f_\gamma<\infty$, this matches both $f_R$ and
$\partial_tf_R$ at $T$, as needed to preserve the H\"older regularity.
If $f(T)=0$, instead set
$\widetilde f_R(T+s,x)=-f_R(T-s,x)$.
For $0<\delta<T/6$, let $v_{R,\delta}$ on $[T/3,T]$ be the
convolution of $\widetilde f_R$ with a fixed nonnegative smooth
probability kernel, even in time,
rescaled to have support in $(-\delta,\delta)\times B(0,\sqrt\delta)$.
Choose a fixed smooth cutoff $\chi=1$ on $[0,T/3]$ and $\chi=0$
on $[2T/3,T]$. Set $w_{R,\delta}=S_\delta f_R$ on $[0,T/3]$,
$w_{R,\delta}=v_{R,\delta}$ on $[2T/3,T]$, and
\begin{equation*}
 w_{R,\delta}=\chi S_\delta f_R+(1-\chi)v_{R,\delta}
 \quad\text{on }[T/3,2T/3].
\end{equation*}
The convolutions define $w_{R,\delta}$ on a neighborhood of $[0,T]$;
a smooth compact time cutoff equal to one on $[0,T]$ makes
$w_{R,\delta}\in C_c^\infty(\R^{d+1})$.
Odd reflection and the even time kernel give $w_{R,\delta}(T)=0$
whenever $f(T)=0$.

We next estimate $E(w_{R,\delta}-f)$.
For fixed $R$, $S_\delta$ commutes with weak differentiation.
Applying $S_\delta F\to F$ to each derivative of $f_R$ in $E$
gives $E_{(0,2T/3)}(S_\delta f_R-f_R)\to0$.
On $[T/6,T]$ and a fixed compact set containing the supports of
$f_R$ and $v_{R,\delta}$ in space, $d\mu_t/dx$ is bounded above
and away from zero, with bounds allowed to depend on $R$ and the
fixed initial law $\mu_0$.
Thus the $L^2(\nu)$ and $L^2(dt\,dx)$ norms are equivalent there.
The convolution argument in \cite[Section~5.3.1]{evans-2010},
applied to each weak derivative in $E$, therefore gives
$E_{(T/3,T)}(v_{R,\delta}-f_R)\to0$.
The product rule and dominated convergence yield
\begin{equation*}
 \lim_{\delta\downarrow0}E(w_{R,\delta}-f_R)=0
 \quad\text{for each fixed }R,\qquad
 \lim_{R\to\infty}E(f_R-f)=0.
\end{equation*}

For convergence in $d_h$, assume the stated initial continuity
and integrability. At each fixed node $t_i>0$,
\cref{eq:energy-trace} on $[0,t_i]$ gives
\begin{equation*}
 \norm{J^2(w_{R,\delta}-f)(t_i)}_{L^2(\mu_{t_i})}
 \le C_{t_i}E(w_{R,\delta}-f).
\end{equation*}
At $t_0=0$, we have
$w_{R,\delta}(0)=S_\delta f_R(0)$.
For fixed $R$, local continuity gives
$J^2w_{R,\delta}(0)\to J^2f_R(0)$ uniformly on a common compact
support, hence in $L^2(\mu_0)$.
Spatial cutoff gives $J^2f_R(0)\to J^2f(0)$ in $L^2(\mu_0)$.

If $\norm f_\gamma<\infty$, spatial cutoff, either terminal
extension and convolution with the nonnegative kernels are
bounded in the parabolic H\"older norm, uniformly in $R,\delta$.
Multiplication by the fixed cutoff $\chi$ gives
\begin{equation*}
 \norm{w_{R,\delta}}_\gamma\le C_\gamma\norm f_\gamma,
 \qquad R\ge1,\quad 0<\delta<T/6.
\end{equation*}

Choose $R_m\to\infty$ and then $\delta_m\downarrow0$ so that
$w_m=w_{R_m,\delta_m}$ satisfies $E(w_m-f)\to0$ and, under the initial assumptions,
$J^2w_m(0)\to J^2f(0)$ in $L^2(\mu_0)$.
\end{proof}

\FloatBarrier
\section*{Acknowledgments}
The authors used ChatGPT (GPT-6-astra) to generate the experimental
code, which they independently organized and checked.
The authors assume responsibility for all content.

\begingroup
\microtypesetup{activate=true}
% Bibliography embedded for a self-contained source file.
\phantomsection

\endgroup

\end{document}